\documentclass[reqno]{amsart}
\usepackage{url}
\usepackage{amssymb,amsthm,amsfonts,amstext,amsmath}
\usepackage{newcent}         
\usepackage{helvet}          
\usepackage{courier}         
\usepackage{graphicx}
\usepackage{enumerate}
\usepackage{hyperref}
\numberwithin{equation}{section}
\usepackage{color}
\usepackage{float}
\usepackage{mathrsfs}
\usepackage{tikz}
\usetikzlibrary{calc}
\definecolor{light-gray}{gray}{0.95}
\usepackage{float}
\newtheorem{theorem}{Theorem}[section]

\newtheorem{lemma}[theorem]{Lemma}
\newtheorem{proposition}[theorem]{Proposition}

\newtheorem{remark}[theorem]{Remark}
\newtheorem{definition}{Definition}

\usepackage{graphicx}
\newcommand{\cev}[1]{\reflectbox{\ensuremath{\vec{\reflectbox{\ensuremath{#1}}}}}}

\newcommand{\mc}[1]{{\mathcal #1}}
\newcommand{\mf}[1]{{\mathfrak #1}}

\newcommand{\bb}[1]{{\mathbb #1}}

\newcommand{\eps}{\varepsilon}

\newcommand{\<}{\langle}
\renewcommand{\>}{\rangle}
\newcommand{\p}{\partial}
\newcommand{\pfrac}[2]{\genfrac{}{}{}{1}{#1}{#2}}

\newcommand{\dd}{\displaystyle}

\newcommand{\Ln}{\mf L^{{\mbox{\rm \scriptsize Neu}}}}
\newcommand{\Lr}{\mf L^{{\mbox{\rm \scriptsize Rob}}}}
\newcommand{\Dn}{\mf D^{{\mbox{\rm \scriptsize Neu}}}}
\newcommand{\Dr}{\mf D^{{\mbox{\rm \scriptsize Rob}}}}

\newcommand{\AF}{\mathcal{A}_{\text{\rm Fried}}}

\newcommand{\DF}{\mf D_{\text{\rm Fried}}}

\newcommand{\Hn}{\mathscr{H}^{{\mbox{\rm \scriptsize Neu}}}}
\newcommand{\HF}{\mathscr{H}_{\text{\rm Fried}}}
\newcommand{\HrF}{\mathscr{H}_{\text{\rm Fried}}^{{\mbox{\rm \scriptsize Rob}}}}
\newcommand{\HnF}{\mathscr{H}_{\text{\rm Fried}}^{{\mbox{\rm \scriptsize Neu}}}}

\newcommand{\Gmenos}{\Gamma_{N,-}}
\newcommand{\Gmais}{\Gamma_{N,+}}

\newcommand{\Gmenosj}{\Gamma_{N,-}^j}
\newcommand{\GmenosjR}{\Gamma_{N,-}^{j,\text{\rm right}}}
\newcommand{\GmenosjL}{\Gamma_{N,-}^{j,\text{\rm left}}}

\newcommand{\Gmaisj}{\Gamma_{N,+}^j}
\newcommand{\GmaisjR}{\Gamma_{N,+}^{j,\text{\rm right}}}
\newcommand{\GmaisjL}{\Gamma_{N,+}^{j,\text{\rm left}}}

\newcommand{\Gmenosi}{\Gamma_{N,-}^i}

\newcommand{\uu}{{\bf u}}

\newcommand\topo[2]{\genfrac{}{}{0pt}{}{#1}{#2}}

\newcommand{\sobH}{L^2\big([0,T];\mathcal{H}^1(\bb T^d\backslash \partial\Lambda)\big)}

\def\centerarc[#1](#2)(#3:#4:#5){\draw[#1] ($(#2)+({#5*cos(#3)},{#5*sin(#3)})$) arc (#3:#4:#5);}

\let\oldtocsection=\tocsection
\let\oldtocsubsection=\tocsubsection
\let\oldtocsubsubsection=\tocsubsubsection
\renewcommand{\tocsection}[2]{\hspace{0em}\oldtocsection{#1}{#2}}
\renewcommand{\tocsubsection}[2]{\hspace{1em}\oldtocsubsection{#1}{#2}}
\renewcommand{\tocsubsubsection}[2]{\hspace{2em}\oldtocsubsubsection{#1}{#2}}
\DeclareRobustCommand{\SkipTocEntry}[5]{}


\begin{document}

\title[Hydrodynamic Limit for the Exclusion Process with a Slow Membrane]{Hydrodynamic Limit for the SSEP\\ with a Slow Membrane}

\author[T. Franco]{Tertuliano Franco}
\address{UFBA\\
 Instituto de Matem\'atica, Campus de Ondina, Av. Adhemar de Barros, S/N. CEP 40170-110\\
Salvador, Brazil}
\curraddr{}
\email{tertu@ufba.br}
\thanks{}

\author[M. Tavares]{Mariana Tavares}
\address{UFBA\\
 Instituto de Matem\'atica, Campus de Ondina, Av. Adhemar de Barros, S/N. CEP 40170-110\\
Salvador, Brazil}
\curraddr{}
\email{tavaresaguiar57@gmail.com}
\thanks{}


\begin{abstract} In this paper we consider a symmetric simple exclusion process (SSEP) on the  $d$-dimensional discrete torus $\bb T^d_N$ with a spatial 
non-homogeneity given by a slow membrane. The slow membrane is  defined here as the boundary of a smooth simple connected region $\Lambda$ on the continuous $d$-dimensional torus $\bb T^d$. In this setting, bonds crossing the membrane have jump rate $\alpha/N^\beta$ and all other  bonds have jump rate one, where $\alpha>0$, $\beta\in[0,\infty]$, and $N\in \bb N$ is the scaling parameter. In the diffusive scaling we prove that the hydrodynamic limit presents a dynamical phase transition, that is, it depends on the regime of $\beta$.  For $\beta\in[0,1)$, the hydrodynamic equation is given by the usual heat equation on the continuous torus, meaning that the slow membrane has no effect in the limit. For $\beta\in(1,\infty]$, the hydrodynamic equation  is the heat equation with Neumann boundary conditions, meaning that the slow membrane $\p \Lambda$ divides $\bb T^d$ into two isolated regions $\Lambda$ and $\Lambda^\complement$. And for the critical value $\beta=1$, the hydrodynamic equation is the heat equation with certain Robin boundary conditions related to the Fick's Law. 
\end{abstract}

\maketitle


\allowdisplaybreaks

\section{Introduction}\label{s1}
A central question of Statistical Mechanics is about how microscopic interactions determine the macroscopic behavior of a given system. Under this guideline, an entire area on scaling limits of interacting random particle systems has been developed, see \cite{kl} and references therein. 

In the last years, many attention has been given to scaling limits of (spatially) non-homogeneous interacting systems, see for instance \cite{Franco2010,fgn1} among many others. Such an attention is quite natural due to the fact that a  non-homogeneity may represent vast physical situations, as impurities, changing of density in the media etc. Among those interacting particles systems,  processes of \textit{exclusion type} have special importance: they are, at same time, mathematically tractable and have a physical interaction, leading to precise representation of many phenomena. Being more precise,  a random process is called of \textit{exclusion type} if it has the \textit{hard-core interaction}, that is, at most one particle is allowed per site of a given graph. The random evolution of the  system (in the symmetric case) can be described as follows: to  each edge of the given graph, a Poisson clock is associated, all of them independent. At a ring time of some clock, the occupation values for the vertexes of the  corresponding edge are interchanged.

In \cite{Franco2010}, a quite broad setting for the one-dimensional symmetric exclusion process (SEP) in non-homogeneous medium  has been considered, being obtained its hydrodynamic limit, that is, the law of large numbers for the time evolution of the spatial density of particles. The hydrodynamic equation there was given by a  PDE related to a  Krein-Feller operator. And in \cite{FARFAN2010}, the fluctuations for the same model were obtained.

The scenario for the SEP in non-homogeneous medium in dimension $d\geq 2$ up to now is far less understood.  In \cite{valentim2012}, a generalization of \cite{Franco2010} to the $d$-dimensional setting was reached. However, the definition of model there was  very specific to permit a reduction to the one-dimensional approach of \cite{Franco2010}.

In \cite{hld}, the hydrodynamic limit in the diffusive scaling for the following $d$-dimensional simple symmetric  exclusion process (SSEP) in non-homogeneous medium was proved, where the term \textit{simple} means that only jumps to nearest neighbors are allowed. The underlying graph is the discrete $d$-dimensional torus, and all bonds of the graph have rate one, except those laying over a $(d-1)$-dimensional closed surface, which have rate given by $N^{-1}$ times a constant depending on the angle between the edge and the normal vector to the surface, where $N$ is the scaling parameter. The hydrodynamic equation obtained  was given by a PDE related to  a $d$-dimensional Krein-Feller operator. Despite less broad in certain sense than the setting of \cite{valentim2012}, the model in \cite{hld} cannot be approached by one-dimensional techniques, being   truly $d$-dimensional. \vspace{-0.3cm}
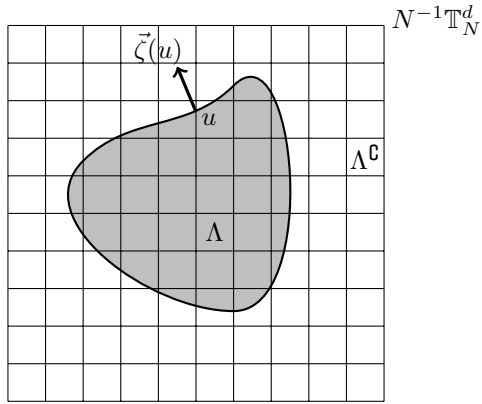
\begin{figure}[htb!]
\centering
\begin{tikzpicture}
\tikzstyle{ponto}=[fill=blue,circle,scale=.25]
\begin{scope}[yshift=0.2cm]
\draw[draw=black,thick, fill=lightgray] (1,2) to[out=45,in=225] (3,3) to[out=45,in=0] (3,0) to[out=180,in=225] (1,2);
\draw[very thick, ->] (2.5,2.65)--(2.25,3.25);
\draw (2,3.5) node{$\vec{\zeta}(u)$};
\draw (2.45,2.75) node[below right]{$u$};
\end{scope}

\draw (2.75,1.25) node {$\Lambda$};
\draw (4.75,2.25) node {$\Lambda^\complement$};

\draw[color=black,step=0.5cm] (0,-1) grid (5,4);
\draw (5.7,3.75) node[above]{$N^{-1}\bb T^d_N$};
\end{tikzpicture}
\caption{The  region in gray represents $\Lambda$, and the white region represents its complement $\Lambda^\complement$. The grid represents $N^{-1}\bb T_N^d$,  the discrete torus embedded on the continuous torus~$\bb T^d$. By $\vec{\zeta}(u)$ we denote the normal exterior unitary vector to $\Lambda$ at the point $u\in \p\Lambda$.}
\label{Fig1}
\end{figure}

In the present paper, we consider a $d$-dimensional model  close to the one in \cite{hld} and related to the \textit{slow bond phase transition behavior} of \cite{fgn1,fgn2,fgn3}. It is fixed a $(d-1)$-dimensional smooth surface $\p\Lambda$ in the continuous $d$-dimensional torus $\bb T^d$, see Figure~\ref{Fig1}. Edges have rates equal to one, except those intersecting $\p\Lambda$, which have rate $\alpha/N^\beta$, where $\alpha>0$, $\beta\in[0,\infty]$ and $N\in \bb N$ is the scaling parameter. Here we prove the hydrodynamic limit, which depends on the range of $\beta$, namely, if $\beta\in[0,1)$, $\beta=1$ or $\beta\in(1,\infty]$.

For $\beta\in[0,1)$, the hydrodynamic equation is given by the usual heat equation:
meaning that, in this regime, the slow bonds do not have any effect in the  continuum limit.
 For $\beta\in (1,\infty]$, the hydrodynamic equation is the  heat equation with the following Neumann boundary conditions over $\p\Lambda$:
\begin{equation*}
\frac{\partial\rho(t,u^{+})}{\partial\vec{\zeta}(u)}=\frac{\partial\rho(t,u^{-})}{\partial\vec{\zeta}(u)}=0,\qquad \forall\, t\geq0, \, u\in \p\Lambda, 
 \end{equation*}
where $\vec{\zeta}$ is the normal unitary vector to $\p\Lambda$. This means that, in this regime, the slow bonds are so strong that there no flux of mass through $\p\Lambda$ in the continuum, despite the existence of flux of particles in the discrete for each $N\in \bb N$. 
For the critical value $\beta=1$, the hydrodynamic equation is given by the  heat equation with the following Robin boundary conditions:
\begin{equation}\label{BRcond}
 \frac{\partial\rho(t,u^{+})}{\partial\vec{\zeta}(u)}=\frac{\partial\rho(t,u^{-})}{\partial\vec{\zeta}(u)}=\alpha\Big(\rho(t,u^{+})-\rho(t,u^{-})\Big)\sum_{j=1}^{d}|\<\vec{\zeta}(u),e_j\>|, \quad t\geq 0, \, u\in \p\Lambda\,, 
\end{equation}
where $u^-$ denotes the limit towards $u\in\p\Lambda$ through points over $\Lambda$ while
 $u^+$ denotes the limit towards $u\in\p\Lambda$ through points over $\Lambda^\complement$, and $\{e_1\ldots,e_d\}$ is the canonical basis of $\bb R^d$.
 
We observe that the Robin boundary condition  above is in agreement with the \textit{Fick's Law}: the spatial derivatives are equal due to the conservation of particles, representing the rate at which  the mass crosses the boundary. Such a  rate is proportional to the difference of concentration on each side of the boundary, being the 
diffusion coefficient through the boundary at a point $u\in \p\Lambda$  given by  $D(u)=\alpha\sum_{j=1}^{d}|\<\vec{\zeta}(u),e_j\>|$. Since $\vec{\zeta}(u)$ is a unitary vector,  the reader can  check via Lagrange multipliers that this diffusion coefficient satisfies
\begin{align*}
\alpha\;\leq \; D(u)\;\leq \; \alpha\sqrt{d}
\end{align*}
in dimension $d\geq 2$. Moreover, in this case $\beta=1$, the hydrodynamic equation  exhibits the phenomena of \textit{non-invariance for isometries}. Let us explain this notion. Consider an isometry $\mathbf{T}:\bb T^d\to\bb T^d$, an initial density profile $\rho_0:\bb T^d\to [0,1]$ and denote by $(S(t)\rho_0)(u)$ the solution of the usual heat equation with initial condition $\rho_0$. Then, 
$$\big(S(t)(\rho_0\circ \mathbf{T})\big)(u)\;=\; (S(t)\rho_0)\big(\mathbf{T}(u)\big)\,.$$
In other words, if we isometrically move the initial condition of the usual heat equation, the solution of the PDE under this new initial condition is the equal to the previous solution moved by the same isometry. On the other hand, as we can see in \eqref{BRcond}, the diffusion coefficient $D(u)$ depends on how the surface $\p\Lambda$ is positioned with respect to the canonical basis.  Hence the PDE for $\beta=1$ \textit{is not invariant for isometries}, differently from the cases $\beta\in[0,1)$ and $\beta\in(1,\infty]$. Note that the diffusion coefficient also says that the underlying graph plays a role in the limit.

Besides the dynamical phase transition itself, this work has the following features.
 First of all, in contrast with some previous works, the hydrodynamic equations are characterized as classical PDEs, with clear interpretation.   In the regime $\beta\in[0,1)$, the proof relies on a sharp replacement lemma which compares occupations of neighbor sites in opposite sides of $\p\Lambda$. For $\beta=1$, the proof is based on a precise analysis of the surface integrals and the model drops the \textit{ad hoc}  hypothesis adopted in \cite{hld}: here the rates for bonds crossing $\p\Lambda$ are all equal to $\alpha/N$, with no extra constant depending on the incident angle.  Finally, a remark the uniqueness of weak solutions for the cases $\beta=1$ and $\beta\in(1,\infty]$. Uniqueness of weak solutions are in general a delicate and technical issue, specially for dimension higher than one. In Proposition~\ref{prop72} we provide a general statement which leads to the  uniqueness of weak solutions in both cases $\beta=1$ and $\beta\in(1,\infty]$. The keystone of the proof is the notion of \textit{Friedrichs extension} for strongly monotone symmetric operators. The uniqueness statement has the feature of being simple, $d$-dimensional and easily adaptable to many contexts.  However, it is strictly limited to the uniqueness of weak solutions of parabolic linear PDEs with linear boundary conditions.
 
 The paper is divided as follows: In Section~\ref{s2} we state definitions and results. In Section~\ref{s3} we draw the strategy of proof for the hydrodynamic limit. In Section~\ref{s4} is reserved to the proof of tightness of the processes. In Section~\ref{s5} we prove the necessary replacement lemmas and energy estimates. In Section~\ref{s6} we characterize limit points as concentrated on weak solutions of the respective PDEs, and in Section~\ref{s7} we assure uniqueness of those weak solutions.

\section{Definitions and Results}\label{s2}

Let $\bb T^d$ be the continuous $d$-dimensional torus, which is $[0,1)^d$ with periodic boundary conditions, and let $\bb T^d_N$ be the discrete torus with $N^d$ points, which can naturally embedded in the continuous torus as $N^{-1} \bb T^d_N$, see Figure~\ref{Fig1}. We therefore will not 
distinguish notation for functions defined on $\bb{T}^d$ or  $ N^{-1} \bb T_{N}^d$.

By $\eta = (\eta(x))_{x \in \bb T_{N}^d}$  we denote configurations in the state space $\Omega_N = \{0,1\}^{\bb T_{N}^d}$, where $\eta(x)=0$ means that the site $x$ is empty, and $\eta(x)=1$ means that the site $x$ is occupied. 
By a \textit{symmetric simple exclusion process} we mean the 
Markov Process with configuration space $\Omega_N$ and exchange rates $\xi^{N}_{x,y} > 0$ for $x,y \in \bb T^d_N$ with $\Vert x-y\Vert_{1} = 1$. This process can be characterized in terms of the infinitesimal generator $\mathscr{L}_{N}$ acting on functions $f:\Omega_N \to \bb{R}$ as
\begin{equation*}\label{ln}
(\mathscr{L}_{N}f)(\eta)\;=\;\sum_{x\in \bb T_{N}^d}\sum_{j=1}^d \,\xi^{N}_{x,x+e_j}\,\Big[f(\eta^{x,x+e_j})-f(\eta)\Big]\,,
\end{equation*}
where $\{e_1,\ldots,e_d\}$ is the canonical basis of $\bb R^d$ and $\eta^{x,x+e_j}$ is the configuration obtained from $\eta$ by exchanging the occupation variables 
$\eta(x)$ and $\eta(x+e_j)$, that is,
\begin{equation*}
\eta^{x,x+e_j}(y)\;=\;\left\{\begin{array}{cl}
\eta(x+e_j),& \mbox{if}\,\,\, y=x\,,\\ 
\eta(x),& \mbox{if} \,\,\,y=x+e_j\,,\\ 
\eta(y),& \mbox{otherwise.}
\end{array}
\right.
\end{equation*}
The Bernoulli product measures $\{\nu^N_\theta\,:\,\theta \in [0,1]\}$ are invariant and in fact, reversible, for the symmetric nearest neighbor exclusion process introduced above. Namely, $\nu^N_\theta$ is a product measure on $\Omega_N$ whose marginal at site $x\in \bb T^d_N$ is given by
\begin{equation*}
\nu^N_\theta\{ \eta:\eta(x)=1\}\;=\;\theta\,.
\end{equation*}

Fix now  two parameters $\alpha>0$ and $\beta \in [0,\infty]$ and a simple connected closed region $\Lambda\subset\bb T^d$ whose  boundary $\p\Lambda$ is a smooth $(d-1)$-dimensional surface. 
The \textit{symmetric simple exclusion process with slow bonds over} $\p\Lambda$ (SSEP with slow bonds over $\p\Lambda$) we define now is the particular 
simple symmetric  exclusion process with exchange rates  given by
\begin{equation}
\label{rates}
\xi^N_{x,x+e_j}\;=\;\left\{\begin{array}{cl}
\dfrac{\alpha}{N^\beta}\,, &  \mbox{if }~\dfrac{x}{N}\in\Lambda\textrm{ and }
\frac{x+e_j}{N}\in\Lambda^\complement, \textrm{ or } \dfrac{x}{N}\in\Lambda^\complement \textrm{ and } \dfrac{x+e_j}{N}\in\Lambda,\medskip\\ 
1\,, &\mbox{otherwise,}
\end{array}
\right.
\end{equation}
for all $x \in \bb T^d_N$ and $j=1,\ldots,d$. That is, the \textit{slow bonds} of the process will be the bonds in $N^{-1}\bb T^{d}_{N}$ for which one of its vertices belongs to $\Lambda$ and the other one belongs to $\Lambda^\complement$. See Figure~\ref{Fig1} for an illustration.

Note that, when $\beta=\infty$, there are no crossings of particles through the boundary $\p\Lambda$. 
From now on, abusing of notation,  we will call the generator of the SSEP with slow bonds over $\p\Lambda$ by $\mathscr{L}_N$, being understood  that jump rates will be given by~\eqref{rates}.

Denote by $\{\eta_t : t\ge 0\}$ the Markov process with state space $\Omega_N$ and generator $N^2\mathscr{L}_N$, where the $N^2$ factor is the so-called \textit{diffusive scaling}. This Markov process depends on $N$, but it will not be indexed on it to not overload notation. Let $D(\bb R_+, \Omega_N)$ be the \textit{Skorohod space} of c\`adl\`ag trajectories taking values in $\Omega_N$. For a measure $\mu_N$ on $\Omega_N$, denote by $\bb P^N_{\mu_N}$ the probability measure on $D(\bb R_+, \Omega_N)$ induced by the initial state $\mu_N$ and the Markov process $\{\eta_t : t\ge 0\}$. Expectation with respect to $\bb P^N_{\mu_N}$ will be denoted by $\bb E^N_{\mu_N}$.

In the sequel, we present the partial differential equations governing the time evolution of the density profile for the different regimes of $\beta$, defining the notion of weak solution for each one of those equations.
Denote by $\rho_t$ a function $\rho(t, \cdot)$ and denote by $C^n(\bb T^d)$ the set of continuous functions from $\bb T^d$ to $\bb R$  with continuous derivatives of order up to $n$.  Let $\< \cdot , \cdot \>$ and $\| \cdot \|$ be the  inner product and norm in  $L^2(\bb T^d)$, that is,
\begin{equation}\label{inner}
\< f, g \> \;=\; \int_{\bb T^d} f(u)\, g(u)\, du\; \, \;\textrm{ and }\; \, 
\| f \| = \sqrt{ \< f, f \> } \; ,\quad \forall\, f,  g \in L^2(\bb T^d)\,.
\end{equation}
Fix once and for all a  measurable density profile $\rho_0: \bb T^d \rightarrow [0,1]$. Note that $\rho_0$ is bounded.
\begin{definition} \label{beta<1}
A bounded function $\rho : [0,T]\times \bb T^d \to \bb R$
is said to be a weak solution of the heat equation
\begin{equation}
\label{edpheat}
\left\{
\begin{array}{ll}
{\displaystyle \p_t \rho(t,u) \; =\; \Delta \rho(t,u)}, & t\geq0, \, u\in \bb T^d, \\
{\displaystyle \rho(0,u) \;=\; \rho_0(u)},&  u\in\bb T^d  .
\end{array}
\right.
\end{equation}
if, for all functions  $H\in C^2(\bb T^d)$ and all $t\in[0,T]$, the function $\rho(t, \cdot)$ satisfies the integral equation
\begin{equation*}\label{eqint1}
\< \rho_t, H\> \;-\; \< \rho_0 , H\> 
- \int_0^t \< \rho_s , \Delta H \> \, ds\; \;=\; 0\,.
\end{equation*}
\end{definition}

We recall next the definition of Sobolev Space from \cite{e}.
Let $U$ be an open set of $\bb R^d$ or $\bb T^d$.
The Sobolev Space $\mc H^{1}(U)$ consists of all locally summable functions $\kappa:U\rightarrow \bb R$ such that there exist functions $\p_{u_{j}}\kappa\in L^{2}(U)$, $j=1,\ldots, d$, satisfying $$\int_{\bb T^d} \p_{u_j}H(u)\kappa(u)\,du\;=\;-\int_{\bb T^d}H(u)\p_{u_{j}}\kappa(u)\,du$$
for all $H\in C^{\infty}(U)$ with compact support. Furthermore, for $\kappa\in \mc H^{1}(U)$, we define the norm 
$\Vert\kappa\Vert_{\mc H^1(U)}\;=\;\Big(\sum_{j=1}^d\int_U \big|\p_{u_{j}}\kappa\big|^2\, du\Big)^{1/2}$. Finally, we define the space $L^2([0,T], \mc H^{1}(U))$, which consists of all measurable functions $\tau:[0,T]\rightarrow \mc H^{1}(U)$ such that
$$\Vert\tau\Vert_{L^2([0,T], \mc H^1(U))}\;:=\; \Big(\int_0^T \Vert \tau_t\Vert^2_{\mc H^1(U)}\, dt\Big)^{1/2}\;<\;\infty\,.$$
 Note  that $U=\bb T^d\backslash \p\Lambda$ is an open subset of $\bb T^d$.\medskip

 The following  notation will be used several times along the text.
Given a function $f:\bb T^d\backslash \p \Lambda\to \bb R$ and $u\in \p\Lambda$, we denote
\begin{equation}\label{maismenos}
f(u^+)\;:=\;\lim_{\topo{v\to u}{v\in \Lambda^\complement}}f(v)\quad \text{ and } f(u^-)\;:=\;\lim_{\topo{v\to u}{v\in \Lambda}}f(v)\,,
\end{equation}
that is, $f(u^+)$ is the limit of $f(v)$ as $v$ approaches $u\in \p\Lambda$ through \textit{the complement of} $\Lambda$, while  $f(u^-)$ is the limit of $f(v)$ as $v$ approaches $u\in \p\Lambda$ through~$\Lambda$. Let  $\mathbf{1}_A$ be the indicator function of a set $A$, that is, $\mathbf{1}_A(a)=1$ if $a\in A$ and zero otherwise.  Denote by $\vec{\zeta}(u)$ the normal unitary exterior vector to the region  $\Lambda$ at the point $u\in\p\Lambda$ and by $\p/\p \vec{\zeta}$ the directional derivative with respect to $\vec{\zeta}(u)$.  

Below, by $\<\vec{u},\vec{v}\>$ we denote the canonical inner product of two vectors $\vec{u}$ and $\vec{v}$ in $\bb R^d$, which shall  not be  misunderstood with the inner product in $L^2(\bb T^d)$ as defined in  \eqref{inner}. By $dS$ we indicate a surface integral.

\begin{definition} \label{beta=1}
A bounded function $\rho : [0,T]\times \bb T^d \to \bb R$
is said to be a weak solution of the following heat equation with Robin boundary conditions
\begin{equation}
\label{edp12}
\left\{
\begin{array}{ll}
{\displaystyle \p_t \rho(t,u) \; =\; \Delta \rho(t,u)}, &\hspace{-0.2cm}  t\geq0, \, u\in \bb T^d, \\
{\displaystyle \frac{\partial\rho(t,u^{+})}{\partial\vec{\zeta}(u)}=\frac{\partial\rho(t,u^{-})}{\partial\vec{\zeta}(u)}=\alpha\Big(\rho(t,u^{+})-\rho(t,u^{-})\Big)\sum_{j=1}^{d}|\<\vec{\zeta}(u),e_j\>|}, &\hspace{-0.2cm} t\geq0, \, u\in \p\Lambda, \\
{\displaystyle \rho(0,u) \;=\; \rho_0(u)}, &\hspace{-0.2cm} u\in\bb T^d \, .
\end{array}
\right.
\end{equation}
 if $\rho\in L^2([0,T], \mc H^1(\bb T^d\backslash \p\Lambda))$ and, for all functions  $H=h_1\mathbf{1}_{\Lambda}+h_2\mathbf{1}_{\Lambda^\complement}$ with $h_1, h_2 \in C^2(\bb T^d)$ and for all $t\in[0,T]$,  the following  the integral equation holds:
\begin{equation*}
\begin{split}
&\< \rho_t, H\> - \< \rho_0 , H\> - \int_0^t\! \< \rho_s , \Delta H \> \, ds-\int_0^t\!\int_{\p\Lambda}\!\rho_s(u^+)\sum_{j=1}^d\p_{u_j} H(u^+)\<\vec{\zeta}(u),e_j\>\,dS(u)ds
\\
&
+\int_0^t\int_{\p\Lambda}\!\rho_s(u^-)\sum_{j=1}^d\p_{u_j} H(u^-)\<\vec{\zeta}(u),e_j\>\,dS(u)ds\\
&+\int_0^t\int_{\p\Lambda}\!\alpha\,(\rho_s(u^-)-\rho_s(u^+))(H(u^+)-H(u^-))\Big(\sum_{j=1}^{d}|\<\vec{\zeta}(u),e_j\>|\Big)\,dS(u)ds\;=\; 0\,.
\end{split}
\end{equation*}
\end{definition}

The reader should note that the function $H$ is (possibly) discontinuous at the boundary $\p\Lambda$. Note also that the expression $\sum_{j=1}^d\p_{u_j} H(u^\pm)\<\vec{\zeta}(u),e_j\>$ appearing in the integral equation above is nothing but $\p H (u^\pm)/\p \vec{\zeta}$ due to  linearity of the directional derivative. 

\begin{definition} \label{beta>1}
A bounded function $\rho : [0,T]\times \bb T^d \to \bb R$
is said to be a weak solution of the heat equation with Neumann boundary conditions
\begin{equation}
\label{edpbc}
\left\{
\begin{array}{ll}
{\displaystyle \p_t \rho(t,u) \; =\; \Delta\rho(t,u)},& t\geq0, \, u\in \bb T^d,\smallskip \\
\displaystyle \frac{\partial\rho(t,u^{+})}{\partial\vec{\zeta}(u)}=\frac{\partial\rho(t,u^{-})}{\partial\vec{\zeta}(u)}=0, & t\geq0, \, u\in \p\Lambda, \smallskip\\
{\displaystyle \rho(0,u) \;=\; \rho_0(u)}, & u\in\bb T^d \,,
\end{array}
\right.
\end{equation}
if $\rho\in L^2([0,T], \mc H^1(\bb T^d\backslash \p\Lambda))$ and, for all functions  $H=h_1\mathbf{1}_{\Lambda}+h_2\mathbf{1}_{\Lambda^\complement}$ with $h_1, h_2 \in C^2(\bb T^d)$ and for all $t\in[0,T]$, the following  integral equation holds:
\begin{equation*}\label{eqint3}
\begin{split}
&\< \rho_t, H\> \;-\; \< \rho_0 , H\> - \!\int_0^t \!\< \rho_s , \Delta H \> \, ds-\int_0^t\!\!\int_{\p\Lambda}\!\rho_s(u^+)\sum_{j=1}^d\p_{u_j} H(u^+)\<\vec{\zeta}(u),e_j\>\,dS(u)ds\\
&+\int_0^t\!\!\int_{\p\Lambda}\rho_s(u^-)\sum_{j=1}^d\p_{u_j} H(u^-)\<\vec{\zeta}(u),e_j\>\,dS(u)ds\;=\;0\,.
\end{split}
\end{equation*}
\end{definition}
Since in Definitions \ref{beta=1} and \ref{beta>1} we impose $\rho \in L^2([0,T], \mc H^1(\bb T^d\backslash\p\Lambda))$, the integrals above are well-defined on the boundary due to the notion of trace in Sobolev spaces, see \cite{e} on the subject. We clarify that the notion of weak solutions above have been defined in the standard way of Analysis: the reader can check that a strong solution of \eqref{edpheat}, \eqref{edp12} or \eqref{edpbc} is indeed a weak solution of the respective PDE.

Fix a  measurable density profile $\rho_0:\bb T^d\rightarrow [0,1]$. For each $N\in\bb N$, let $\mu_N$ be a probability measure on $\Omega_N$. A sequence of probability measures $\{\mu_N \,: \,N\geq 1 \}$ is said to be \textit{associated to a profile} 
$\rho_0 :\bb T^d \to [0,1]$ if, for every $\delta>0$ and every continuous function $H:\bb T^d\rightarrow \bb R$ the following limit holds:
\begin{equation}
\label{profile}
\lim_{N\to\infty}
\mu_N \Bigg\{ \, \Bigg| \frac {1}{N^d}\!\! \sum_{x\in\bb T_N^d} H(x/N) \eta(x)
- \int H(u) \rho_0(u) du \Bigg| > \delta \Bigg\} \;=\; 0\,.
\end{equation}

Below, we establish the main result of this  paper, the hydrodynamic limit for the exclusion process with slow bonds, which depends on the regime of $\beta$.

\begin{theorem}
\label{t01}
Fix $\beta\in[0,\infty]$. Consider the exclusion process with slow bonds over $\p\Lambda$ with rate $\alpha N^{-\beta}$ at each one of these slow bonds. Fix a Borel measurable initial profile $\rho_0 : \bb T^d \to [0,1]$ and consider a sequence of probability measures $\{\mu_N\}_{N\in\bb N}$ on $\Omega_N$ associated to $\rho_0$ in the sense of \eqref{profile}. Then, for each $t\in[0,T]$,
\begin{equation*}
\lim_{N\to\infty} \bb P^N_{\mu_N} \Bigg[\,\eta\,: \, \Bigg| \,\frac{1}{N^d} \sum_{x\in\bb T^d_N} H(x/N)\, \eta_t(x) - \int_{\bb{T}^d} H(u)\, \rho(t,u) du \,\Bigg| \,>\, \delta\, \Bigg] \;=\; 0 \, ,
\end{equation*}
for every $\delta>0$ and every function $H\in C(\bb{T}^d)$ where: \smallskip
\begin{itemize}
	\item If $\beta \in [0,1)$, then $\rho$ is the unique weak solution of \eqref{edpheat}.\smallskip
	\item If $\beta=1$, then $\rho$ is the unique weak solution of \eqref{edp12}.
	\smallskip
	\item If $\beta \in (1,\infty]$, then $\rho$ is the unique weak solution of \eqref{edpbc}.
\end{itemize}
\end{theorem}
The assumption that $\Lambda$ is  simple  and connected may be dropped, being imposed only for the sake of clarity. Otherwise,  notation would be highly overloaded. 

\section{Scaling Limit and Proof's Structure}\label{s3}
Let $\mc M$ be the space of positive Radon measures on $\bb T^d$ with total mass bounded by one, endowed with the weak topology. Let $\pi_t^N\in\mc M$ the empirical measure at time $t$ associated to $\eta_t$, it is a measure on $\bb T^d$ obtained rescaling space by $N$: 
\begin{equation*}
\pi_t^N(du)\;=\;\pi_t^N(\eta_t, du)\;:=\;\frac{1}{N^d}\sum_{x\in\bb T^d_N}\eta_t(x)\delta_{x/N}(du)\,,
\end{equation*}
where $\delta_u$ denotes the Dirac measure concentrated on $u\in \bb T^d$.
For a measurable function $H:\bb T^d\rightarrow\bb R$ which is $\pi$-integrable, denote by $\<\pi_t^N, H\>$ the integral of $H$ with respect to $\pi_t^N$:
\begin{equation*}
\<\pi_t^N, H\>\; =\; \frac{1}{N^d}\sum_{x\in\bb T^d_N}H\left(\pfrac{x}{N}\right)\eta_t(x)\,.
\end{equation*}
Note that this notation $\<\cdot,\cdot\>$ is also used  as the inner product of $L^2(\bb T^d)$. Fix once and for all a time horizon $T>0$. Let $D([0,T], \mc M)$ be the space of $\mc M$-valued
\textit{c\`adl\`ag} trajectories $\pi:[0,T]\to\mc M$ endowed with the
\emph{Skorohod} topology. Then, the $\mc M$-valued process $\{\pi^N_t:t\ge 0\}$
is a random element of $D([0,T], \mc M)$  determined by  $\{\eta_t : t\ge 0\}$. For each probability measure $\mu_N$ on $\Omega_N$, denote by $\bb Q_{\mu_N}^{\beta,N}$ the distribution of $\{\pi^N_t:t\ge 0\}$ on the path space $D([0,T], \mc M)$, when $\eta_0^N$ has distribution $\mu_N$. 

Fix a continuous Borel measurable profile $\rho_0 : \bb T^d \to [0,1]$ and consider a sequence $\{\mu_N:N\geq1\}$ of measures on $\Omega_N$ associated to $\rho_0$. Let $\bb Q^\beta$ be the probability measure on $D([0,T], \mc M)$ concentrated 
on the deterministic path  $\pi(t,du) = \rho (t,u)du$, where:\smallskip
\begin{itemize}
	\item if $\beta \in [0,1)$, then $\rho$ is the unique weak solution of \eqref{edpheat},\smallskip
	\item if $\beta=1$, then $\rho$ is the unique weak solution of \eqref{edp12},
	\smallskip
	\item if $\beta \in (1,\infty]$, then $\rho$ is the unique weak solution of \eqref{edpbc}.
\end{itemize}

\begin{proposition}
\label{s15} For any $\beta\in[0,\infty]$, the sequence of probability measures $\bb Q_{\mu_N}^{\beta,N}$ converges weakly to $\bb Q^{\beta}$ as $N$ goes to infinity.
\end{proposition}

The proof of this result is divided into three parts. In the next section, we show that tightness of  the sequence $\{\bb Q_{\mu_N}^{\beta,N}: N\geq1\}$. In Section~\ref{s5}, we prove a suitable \textit{Replacement Lemma} for each regime of $\beta$, which will be crucial in the task of characterizing limit points. In Section~\ref{s6} we  characterize the limit points of the sequence for each regime of the parameter $\beta$. Finally, the uniqueness of weak solutions is presented in Section \ref{s7} and this implies the uniqueness of limit points of the sequence $\{\bb Q_{\mu_N}^{\beta,N}: N\geq1\}$. 

Finally, we note that Theorem \ref{t01} is a consequence of Proposition \ref{s15}. Actually, since $\bb Q_{\mu_N}^{\beta,N}$ weakly converges  to $\bb Q^{\beta}$ for all continuous functions $H:\bb T^d\rightarrow \bb R$, it follows that the path $\{\<\pi_t^N, H\>:\, 0\leq t\leq T\}$ converges in distribution to $\{\<\pi_t, H\>:\, 0\leq t\leq T\}$. Since $\{\<\pi_t, H\>:\, 0\leq t\leq T\}$  is a deterministic  path, convergence in distribution is equivalent to convergence in probability. Therefore,
\begin{align*}
&\lim_{N\to\infty} \bb P^N_{\mu_N} \Bigg\{ \, \Bigg\vert \frac{1}{N^d} \sum_{x\in\bb T_N^d} H(x/N)\, \eta_t(x) - \int_{\bb{T}^d} H(u) \rho(t,u) du \Bigg\vert 
> \delta \Bigg\}\\
&= \lim_{N\to\infty} \bb Q_{\mu_N}^{\beta,N}\big\{|\<\pi_t^N, H\>-\<\pi_t, H\>|>\delta\big\}\;=\; 0\,, 
\end{align*}
for all $\delta>0$ and  $0\leq t \leq T$.  This gives the strategy of proof for the hydrodynamic limit. Next, we make some general observations.\medskip

Since particles in the exclusion process evolve independently as a nearest neighbor random walk, except for exclusion rule, the exclusion process with slow bonds over $\p\Lambda$ is related to the random walk on $N^{-1}\bb T^d_N$ that describes the evolution of the system with a single particle. To be used throughout the paper we introduce the generator of the random walk described above, which is 
\begin{equation} \label{bbLN}
\bb L_N H\big(\pfrac{x}{N}\big) = \sum_{j=1}^d \Big\{ \xi^N_{x,x+e_j} \, \Big[ H\big(\pfrac{x+e_j}{N}\big) 
- H\big(\pfrac{x}{N}\big) \Big] + \xi^N_{x,x-e_j} \, \Big[H\big(\pfrac{x-e_j}{N}\big) - H\big(\pfrac{x}{N}\big) \Big] \Big\}
\end{equation}
for every $H: N^{-1} \bb T^d_N \rightarrow \mathbb{R}$ and every $x \in \bb T^d_N$. Above, it is understood that $\xi_{x\pm e_j,x}=\xi_{x,x\pm e_j}$.
 By Dynkin's formula (see A.1.5.1 in \cite{kl}),
\begin{equation*}
M^{N}_{t}(H)\;=\;\<\pi^{N}_{t}, H\>- \<\pi^{N}_{0}, H\>-\int_{0}^{t}N^{2} \mathscr{L}_N\<\pi^{N}_{s},H\>ds
\end{equation*}
is a martingale with respect to the natural filtration $\mc F_t:=\sigma(\eta_s^N\,:\, s\leq t)$. By some elementary calculations,
\begin{equation*}
N^2\mathscr{L}_N\<\pi_s^N,H\>\;=\;\frac{1}{N^{d-2}}\sum_{x\in\bb T^d_N}\eta_s(x)\bb L_N H\Big(\frac{x}{N}\Big)\;=\;\<\pi_s^N,N^2\bb L_N H\>\,,
\end{equation*}
hence the martingale can be rewritten as
\begin{equation}\label{M}
M^{N}_{t}(H)\;=\;\<\pi^{N}_{t}, H\>- \<\pi^{N}_{0}, H\>-\int_{0}^{t}\<\pi^{N}_{s},N^2\bb L_N H\>ds\,.
\end{equation}
Note that this observation stands for any jump rates. The particular form of jump rates for the  SSEP with slow bonds over $\p\Lambda$ will play a role when  characterizing limit points and proving replacement lemmas.

\section{Tightness} \label{s4}
This section  deals with the issue of tightness for the sequence  $\{\bb Q_{\mu_N}^{\beta,N}: N\geq1\}$ of probability measures on $D([0,T], \mc M)$.

\begin{proposition}
\label{tight}
For any fixed $\beta\in[0,\infty]$, the sequence of measures $\{\bb Q_{\mu_N}^{\beta,N}: N\geq1\}$ is tight in the Skorohod topology of $D([0,T], \mc M)$.
\end{proposition}

\begin{proof}
In order to prove tightness of $\{\pi_t^N:0\leq t\leq T\}$, it is enough to show tightness of the real-valued process $\{\<\pi_t^N,H\>:0\leq t \leq T\}$ for $H\in C(\bb T^d)$. In fact, (cf. Proposition 1.7, chapter 4 of \cite{kl}) it is enough to show tightness of $\{\<\pi_t^N,H\>:0\leq t \leq T\}$ in $D([0,T],\bb R)$ for a dense set of functions in $C(\bb T^d)$ with respect to the uniform topology. 

For that purpose, fix $H\in C^2(\bb T^d)$. Since the sum of tight processes is tight, in order to prove tightness of $\{\<\pi_t^N, H\>: N\geq1\}$, it is enough to assure  tightness of each term in \eqref{M}. 
The quadratic variation of $M_t^N(H)$ is given by 
\begin{equation} \label{quamar0}
\<M^{N}(H)\>_t=\!\int_{0}^{t}\sum_{j=1}^d\sum_{x\in\bb T_{N}^d} \frac{\xi^{N}_{x,x+e_j}}{N^{2d-2}}\Big[(\eta_{s}(x)-\eta_{s}(x+e_j))(H(\pfrac{x+e_j}{N})-H(\pfrac{x}{N}))\Big]^{2} ds,
\end{equation}
implying that
\begin{equation} \label{quamar}
\<M^{N}(H)\>_t \;\leq\;  \frac{\alpha t}{N^{d}} \, \sum_{j=1}^d\|\p_{u_j}H\|^2_{\infty} \, ,
\end{equation}
where $\|H\|_{\infty}:= \sup_{u\in \bb T^d}|H(u)|$, hence $M^{N}_{t}$ converges to zero as $N\rightarrow\infty$ in $L^{2}(\bb P^\beta_{\mu_N})$. Therefore, by Doob's inequality, for every $\delta>0$,
\begin{equation}\label{limmart}
\lim_{N\rightarrow\infty} \bb P^N_{\mu_N}\Big[\sup_{0\leq t\leq T}|M_t^N(H)|>\delta\Big]\;=\;0\,,
\end{equation}
which  implies tightness of the sequence of martingales $\{M_t^N(H)\,:\,N\geq1\}$. 
Next, we  will prove  tightness for the integral term in \eqref{M}.
Let $\Gamma_N$ be the set of vertices in $\bb T^d_N$ having some incident edge with exchange rate not equal to one, that is,
\begin{align}\label{GammaN}
\Gamma_N=\Big\{x\in \bb T_N^d: \text{ for some } j=1,\ldots,d, \quad\xi^N_{x,x+e_j}=\frac{\alpha}{N^\beta}\text{ or } \xi^N_{x,x-e_j}=\frac{\alpha}{N^\beta}\Big\}.
\end{align}
 The term $\<\pi^{N}_{s},N^2\bb L_N H\>$ appearing inside the time integral in \eqref{M} can be then written as 
\begin{align*}
&\frac{1}{N^{d}}\sum_{j=1}^d\sum_{x\notin \Gamma_N}\eta_{s}(x)N^2\Big[H(\pfrac{x+e_j}{N})+H(\pfrac{x-e_j}{N})-2H(\pfrac{x}{N})\Big] \nonumber\\
&+  \frac{1}{N^{d-1}}\sum_{j=1}^d\sum_{x\in \Gamma_N}\!\eta_{s}(x)\Big[\xi^{N}_{x,x+e_j}N\big(H(\pfrac{x+e_j}{N})\!-\! H(\pfrac{x}{N})\big)\!+\!\xi^{N}_{x,x-e_j}
N\big(H(\pfrac{x-e_j}{N})\!-\!H(\pfrac{x}{N})\big)\!\Big]\label{qqq} 
\end{align*}
since $\xi_{x,x+e_j}=\xi_{x+e_j,x}=1$ for every $x\notin{\Gamma_N}$.
By a Taylor expansion on $H\in C^2(\bb T^d)$, the absolute value of the summand  in the first double sum above is bounded by  $\|\Delta H\|_{\infty}$. Since there are $\mc O(N^{d-1})$ elements in $\Gamma_N$, and
$\xi_{x,x+e_j}\leq \alpha$,  the absolute value of summand in second double sum above is bounded by 
$\sum_{j=1}^d \alpha \Vert\p_{u_j}H\Vert_\infty$.
Therefore, there exists $C>0$, depending only on $H$, such that $|N^2\bb L_N\<\pi_s^N, H\>|\leq C$, which yields
\begin{equation*}
\left|\int_s^t N^2\bb L_N\<\pi_s^N, H\> dr\right|\;\leq\; C|t-s| \,.
\end{equation*}
By \cite[Proposition 4.1.6]{kl}, last inequality implies tightness of the integral term,   concluding the proof of the proposition. 
\end{proof}

\section{Replacement Lemma and Energy Estimates}\label{s5}

This section gives a fundamental result that allow us to replace a mean occupation of a site by the mean density of particles in a small macroscopic box around this site. We start by introducing some tools to be used in the sequel. 

Denote by $H_N(\mu_N|\nu_\theta)$ the relative entropy of $\mu_N$ with respect to the invariant state $\nu_\theta$. For a precise definition and properties of the entropy, we refer the reader to \cite{kl}.  Assuming $0<\theta<1$, the formula  in   \cite[Theorem A1.8.3]{kl}  assures the existence a finite constant $\kappa_0=\kappa_0(\theta)$ such that 
\begin{equation} \label{cte}
H_N(\mu_N|\nu_\theta)\; \leq \; \kappa_0 N^d
\end{equation}
for any probability measure $\mu_N$ on $\{0,1\}^{\bb T^d_N}$.  Denote by $\mf D_N$ the Dirichlet form of the process, which is the functional acting on functions $f:\{0,1\}^{\bb T^d_N}\to \bb R$ as
\begin{equation}\label{Dirichlet}
 \mf D_N(f) \,:=\,\<f,- \mathscr{L}_N f\>_{\nu_\theta} =\sum_{j=1}^d\sum_{x\in \bb T^d_N}\!\frac{\xi^N_{x,x+e_j}}{2}\!\!\int{\!\left({f(\eta^{x,x+e_j})}-{f(\eta)}\right)^2\nu_{\theta}(d\eta)}\,.
\end{equation}
In the sequence,  we will make use of the functional $\mf D_N(\sqrt{f})$, where $f$ is a probability density with respect to $\nu_\theta$.  

\subsection{Replacement Lemma for \texorpdfstring{$\beta\in[0,1)$}{beta<1}}
Below, we define the local density of particles, which corresponds a to the mean occupation in a box around a given site. Abusing of notation, we denote by $\eps N-1$ the integer part of  $\eps N-1$.  For $\beta\in[0,1)$, we define the local mean by
\begin{equation}\label{localmean}
\eta^{\varepsilon N}(x)\;=\;\frac{1}{(\eps N)^d}\sum_{j_1, j_2, \ldots, j_d=0}^{\varepsilon N-1}\eta\left(x+j_1e_1+\ldots+j_de_d\right)\,.
\end{equation}
Note that the sum on the right hand side of above may contain sites in and out of $\Lambda$ in the sense that $x/N\in \Lambda$ or $x/N\in \Lambda^\complement$. By $\mc O(f(N))$ we will mean a  function bounded in modulus by a constant times $f(N)$.
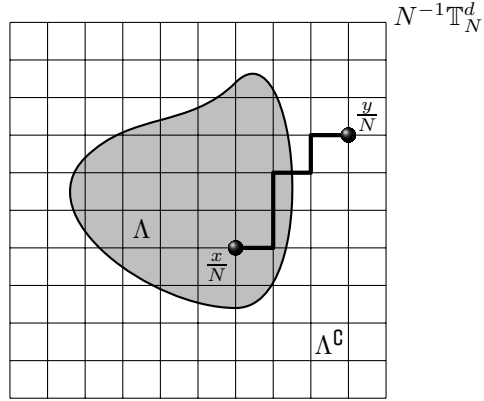
\begin{figure}[htb!]
\centering
\begin{tikzpicture}
\tikzstyle{ponto}=[fill=blue,circle,scale=.25]

\begin{scope}[yshift=0.2cm]
\draw[draw=black,thick, fill=lightgray] (1,2) to[out=45,in=225] (3,3) to[out=45,in=0] (3,0) to[out=180,in=225] (1,2);
\end{scope}

\draw (1.75,1.25) node {$\Lambda$};
\draw (4.25,-0.25) node {$\Lambda^\complement$};

\draw[color=black,step=0.5cm] (0,-1) grid (5,4);
\draw (5.7,3.75) node[above]{$N^{-1}\bb T^d_N$};

\draw[ultra thick,line join=round] (3,1)--(3.5,1)--(3.5,2)--(4,2)--(4,2.5)--(4.5,2.5);

\shade[ball color=black](3,1) circle (0.1);
\shade[ball color=black](4.5,2.5) circle (0.1);

\draw (2.75,0.75) node {$\frac{x}{N}$};
\draw (4.75,2.75) node {$\frac{y}{N}$};
\end{tikzpicture}
\caption{Illustration (in dimension $2$) of a polygonal path joining the sites $x$ and $y=x+j_1e_1+j_2 e_2$, with $j_1=j_2=3$. Note the embedding in the continuous torus $\bb T^d$.}
\label{Fig2}
\end{figure}

\begin{lemma}\label{r1}
Fix $\beta\in [0,1)$. Let $f$ be a density with respect to the invariant measure $\nu_\theta$, $\lambda_N:\bb T_N^d\to\bb R$ a function such that $\Vert \lambda_N\Vert_\infty\leq M<\infty$ and $\gamma>0$. Then, 
\begin{align*}
&\int \gamma N \sum_{x\in \Gamma_N}\lambda_N(x)\big\{\eta(x)-\eta^{\varepsilon N}(x)\big\}f(\eta)\nu_\theta(d\eta)\\
&\leq\; \frac{\gamma^2 M^2 \mc O(N^{d})}{2}\Big(\frac{N^{\beta-1}}{\alpha}+d\eps\Big)+N^2\mf D_N(\sqrt{f}) \, .
\end{align*}
\end{lemma}

\begin{proof}
By the definition  \eqref{localmean} of  local mean $\eta^{\varepsilon N}(x)$, 
\begin{align} 
&\int  \lambda_N(x)\Big\{\eta(x)-\eta^{\varepsilon N}(x)\Big\}f(\eta)\nu_\theta(d\eta)\;=\;\nonumber\\
&= \int\lambda_N(x)\frac{1}{\varepsilon^d N^d}\sum_{j_1, \ldots, j_d=0}^{\varepsilon N-1}\Big\{\eta(x)-\eta(x+j_1e_1+\ldots+j_de_d)
\Big\}f(\eta)\nu_\theta(d\eta)\,.\label{eq51}
\end{align}
The next step is to write  $\eta(x)-\eta(x+j_1e_1+\cdots+j_d e_d)$ as a telescopic sum:
\begin{equation*}
\eta(x)-\eta(x+j_1e_1+\ldots+j_d e_d)\; =\; \sum_{\ell=1}^{j_1+\cdots+j_d}
\eta(a_{\ell-1})-\eta(a_{\ell})\,,
\end{equation*}
where $a_0=x$, $a_{j_1+\cdots+j_\ell}=x+j_1e_1+\cdots+j_d e_d$, and  $\Vert a_{\ell-1}-a_\ell\Vert_1=1$ for any $\ell=1,\ldots, j_1+\cdots+j_d$. Note that the path $a_0,a_1,\ldots,a_{j_1+\cdots+j_\ell}$ depends on the initial point $x$ and the final point $x+j_1e_1+\cdots+j_d e_d$. See Figure~\ref{Fig2} for an illustration and keep in mind that the length of this path is bounded by $d\eps N$.
Inserting the previous equality  into \eqref{eq51}, we get
\begin{align*}
& \int\lambda_N(x)\frac{1}{(\varepsilon N)^d}\sum_{j_1, \ldots, j_d=0}^{\varepsilon N-1}\Big\{\sum_{\ell=1}^{j_1+\cdots+j_d}
\eta(a_{\ell-1})-\eta(a_{\ell})\Big\}f(\eta) \, \nu_\theta(d\eta)\,. 
\end{align*}
Rewriting the expression above as twice the half and performing the transformation $\eta\mapsto  \eta^{a_{\ell-1},a_\ell}$ for which the probability measure $\nu_\theta$ is invariant, expression above becomes: 
\begin{equation*} \label{eq555}
\frac{1}{2(\varepsilon N)^d}\sum_{j_1, \ldots, j_d=0}^{\varepsilon N-1}\sum_{\ell=1}^{j_1+\cdots+j_d}\int\lambda_N(x) \left(\eta(a_{\ell-1})-\eta(a_\ell)\right)\left(f\left(\eta^{a_\ell, a_{\ell-1}}\right)-f\left(\eta\right)\right) \, d\nu_\theta \, .
\end{equation*}
Since  $ab=\sqrt{c}a\pfrac{b}{\sqrt{c}}\leq\pfrac{1}{2}ca^2+\pfrac{1}{2}\pfrac{b^2}{c}$, which holds for any $c>0$, the previous expression is smaller or equal  than
\begin{equation*}
\begin{split}
& \frac{1}{2(\varepsilon N)^d} \sum_{j_1, \ldots, j_d=0}^{\varepsilon N-1}\sum_{\ell=1}^{j_1+\cdots+j_d}\Bigg[\frac{\xi^N_{a_{\ell-1}, a_{\ell}}}{2A}\int\left(\sqrt{f\left(\eta^{a_\ell, a_{\ell-1}}\right)}-\sqrt{f\left(\eta\right)}\right)^2 d\nu_\theta\\
&+\frac{A}{2\xi^N_{a_{\ell-1}, a_\ell}}\int \lambda_N^2(x) \left(\eta(a_\ell)-\eta(a_{\ell-1})\right)^2\Big(\sqrt{f\left(\eta^{a_\ell, a_{\ell-1}}\right)}+\sqrt{f\left(\eta\right)}\Big)^2  d\nu_\theta\Bigg]\,.
\end{split}
\end{equation*}
Summing over $x\in \Gamma_N$, we can bound the last expression by
\begin{equation*}
\begin{split}
& \frac{1}{2(\varepsilon N)^d} \sum_{x\in\Gamma_N}\sum_{j_1, \ldots, j_d=0}^{\varepsilon N-1}\sum_{\ell=1}^{j_1+\cdots+j_d}\Bigg[\frac{\xi^N_{a_{\ell-1}, a_{\ell}}}{2A}\int\left(\sqrt{f\left(\eta^{a_\ell, a_{\ell-1}}\right)}-\sqrt{f\left(\eta\right)}\right)^2 d\nu_\theta\\
&+\sum_{x\in\Gamma_N}\frac{A}{2\xi^N_{a_{\ell-1}, a_\ell}}\int \lambda_N^2(x) \left(\eta(a_\ell)-\eta(a_{\ell-1})\right)^2\Big(\sqrt{f\left(\eta^{a_\ell, a_{\ell-1}}\right)}+\sqrt{f\left(\eta\right)}\Big)^2  d\nu_\theta\Bigg]\,.
\end{split}
\end{equation*}
Recalling \eqref{Dirichlet}, we can bound the first parcel in the sum above by
\begin{align*}
&\frac{1}{2(\varepsilon N)^d}\sum_{j_1, \ldots, j_d=0}^{\varepsilon N-1}\frac{1}{A}\mf D_N(\sqrt{f}) \;=\; \frac{1}{2A}\mf D_N(\sqrt{f}) \,.
\end{align*} 
Since $f$ is a density and  $|\lambda_N(x)|\leq M$, the second parcel is bounded by
\begin{align*} \label{eeee}
\begin{split}
&\frac{1}{2(\varepsilon N)^d}\sum_{x\in\Gamma_N}\sum_{j_1, \ldots, j_d=0}^{\varepsilon N-1}\sum_{\ell=1}^{j_1+\cdots+j_d}  \frac{A}{2}\cdot \frac{4M^2}{\xi^N_{a_{\ell-1}, a_\ell}}\\
&\leq\;\frac{1}{(\varepsilon N)^d}\sum_{j_1, \ldots, j_d=0}^{\varepsilon N-1}AM^2\mc O(N^{d-1})\Big(\frac{N^\beta}{\alpha}+d\varepsilon N\Big)\\
&=\; AM^2\mc O(N^{d-1})\Big(\frac{N^\beta}{\alpha}+d\varepsilon N\Big) \,.
\end{split}
\end{align*}
Up to here we have achieved that 
\begin{align*}
&\int \sum_{x\in \Gamma_N}\lambda_N(x)\big\{\eta(x)-\eta^{\varepsilon N}(x)\big\}f(\eta)\nu_\theta(d\eta)\\
&\;\leq\; AM^2 \mc O(N^{d-1})\Big(\frac{N^\beta}{\alpha}+d\varepsilon N\Big)+\frac{1}{2A}\mf D_N(\sqrt{f}) \, .
\end{align*}
We point out that the quantity of sites on $\Gamma_N$ is of order $\mc O(N^{d-1}$), which is  a consequence of the fact that $\p\Lambda$ is a smooth surface of dimension $d-1$. Then, multiplying the inequality above by $\gamma N$ gives us 
\begin{align*}
&\int \gamma N\sum_{x\in \Gamma_N}\lambda_N(x)\big\{\eta(x)-\eta^{\varepsilon N}(x)\big\}f(\eta)\nu_\theta(d\eta)\\
& \;\leq\; A\gamma \mc O(N^{d}) M^2\Big[\frac{N^\beta}{\alpha}+d\varepsilon N\Big]+\frac{\gamma N}{2A}\mf D_N(\sqrt{f}) \,. 
\end{align*}
Now choosing $A=\gamma N^{-1}/2$ the proof ends. 
\end{proof}

Recall the definition of $\Gamma_N$ in \eqref{GammaN}.

\begin{lemma}[Replacement lemma] \label{replacement} 
Fix $\beta\in [0,1)$. Let $\lambda_N:\bb T_N^d\to\bb R$ be a   sequence of  functions such that $\Vert \lambda_N\Vert_\infty\leq M<\infty$. Then,
\begin{equation*}
\overline{\lim_{\varepsilon\rightarrow0}} \varlimsup_{N\rightarrow\infty}  \bb E^\beta_{\mu_N}\Big[\,\Big|\int_0^t \frac{1}{N^{d-1}} \sum_{x\in\Gamma_N} \lambda_N(x) \{\eta_s^{\varepsilon N}(x)-\eta_s(x)\}\,ds\,\Big|\,\Big]\;=\;0 \, .
\end{equation*}
\end{lemma}

\begin{proof}
Using the variational formula for entropy, for any $\gamma\in\bb R$ (which will be chosen large \textit{a posteriori}), 
\begin{align}
&\bb E^\beta_{\mu_N}\Big[\,\Big|\int_0^t\frac{1}{N^{d-1}} \sum_{x\in\Gamma_N}\lambda_N(x)\{\eta_s(x)-\eta_s^{\varepsilon N}(x)\}ds\Big|\,\Big] \notag \\
&=\frac{1}{\gamma N^{d}}\bb E^\beta_{\mu_N}\Big[\gamma N\,\Big|\int_0^t\sum_{x\in\Gamma_N}\lambda_N(x)\{\eta_s(x)-\eta_s^{\varepsilon N}(x)\}ds\Big|\Big] \notag \\
& \label{ent} \leq  \frac{H_N(\mu_N|\nu_\theta)}{ \gamma N^d}  +   \frac{1}{\gamma N^d} \log \bb E_{\nu_\theta}\Big[\exp\Big(\gamma N\Big|\int_0^t\sum_{x\in\Gamma_N}\lambda_N(x)\{\eta_s(x)-\eta_s^{\varepsilon N}(x)\}ds\Big|\Big)\Big] . 
\end{align}  
By the estimate \eqref{cte} on the entropy,  the first parcel of above is negligible  as $N\rightarrow\infty$ since we will  choose $\gamma$ arbitrarily large. Therefore, we can focus on the second parcel. 
Using that  $e^{|x|}\leq e^x+e^{-x}$ and 
\begin{equation}\label{A.2}
\varlimsup_{N\rightarrow\infty}\frac{1}{N^d}\log (a_N+b_N)\;=\;\max \Big\{\varlimsup_{N\rightarrow\infty}\frac{1}{N^d}\log a_N, \, \varlimsup_{N\rightarrow\infty}\frac{1}{N^d} \log b_N\Big\}
\end{equation}
for any sequences $a_N,b_N>0$, one can see that the second parcel  on the right hand  side of \eqref{ent} is less than or equal to the sum of 
\begin{equation}\label{eq57b}
\varlimsup_{N\rightarrow\infty}\frac{1}{ \gamma N^d}\log\Big\{ \bb E_{\nu_\theta}\Big[\exp\Big(\gamma N\int_0^t\sum_{x\in\Gamma_N}\lambda_N(x)\{\eta_s(x)-\eta_s^{\varepsilon N}(x)\}ds\Big)\Big]\Big\}
\end{equation}
and 
\begin{equation}\label{eq58b}
\varlimsup_{N\rightarrow\infty}\frac{1}{ \gamma N^d}\log\Big\{ \bb E_{\nu_\theta}\Big[\exp\Big(-\gamma N\int_0^t\sum_{x\in\Gamma_N}\lambda_N(x)\{\eta_s(x)-\eta_s^{\varepsilon N}(x)\}ds\Big)\Big]\Big\} \, . 
\end{equation}
We handle  only \eqref{eq57b}, being \eqref{eq58b} analogous.
By Feynman-Kac's formula, see \cite[Appendix 1, Lemma 7.2]{kl}, expression \eqref{eq57b} is bounded by
\begin{equation*}
\varlimsup_{N\rightarrow\infty}\frac{1}{\gamma N^d} \log\Big\{ \exp\Big(\int_0^t \Phi_N\, ds\Big)\Big\}\;=\; \varlimsup_{N\rightarrow\infty}\frac{t\,\Phi_N^1 }{\gamma N^{d}}\, ,
\end{equation*}
where 
\begin{align*}
\Phi_N^1 
&=\sup_{f \ \textrm{density}}\left\{\int\gamma N\sum_{x\in\Gamma_N}\lambda_N(x)\{\eta(x)-\eta^{\varepsilon N}(x)\}f(\eta)\nu_\theta(d\eta)-N^2\mf D_N(\sqrt{f})\right\}\,.
\end{align*}
Applying Lemma \ref{r1} finishes  the proof.  \end{proof}

\subsection{Replacement Lemma for \texorpdfstring{$\beta\in [1,\infty]$}{beta>1}} \label{subsec5.2}
Here, some additional notation  is required. The idea is actually very simple: the local mean shall be over a region avoiding slow bonds. Let $B_N[x,\ell]\subset \bb T^d_N$ be the discrete box centered on $x\in \bb T^d_N$ which edge has size $2\ell$, that is,
$B_N[x,\ell]= \{y\in \bb T^d_N: \Vert y-x\Vert_\infty\leq \ell\}$,
where we have  written $\Vert \cdot \Vert_\infty$ for the supremum norm on $\bb T^d_N$, that is, $\Vert (x_1,\ldots,x_d)\Vert_\infty = \max\big\{|x_1|\wedge |N-x_1|,\ldots,|x_d|\wedge |N-x_d|\big\}$. 
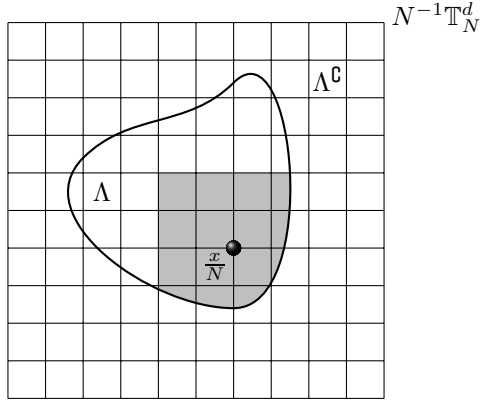
\begin{figure}[H]
\centering
\begin{tikzpicture}
\begin{scope}
  \clip (1,2.2) to[out=45,in=225] (3,3.2) to[out=45,in=0] (3,0.2) to[out=180,in=225] (1,2.2);
  \fill[fill=lightgray] (2,0) rectangle (4,2);
\end{scope}
\begin{scope}[yshift=0.2cm]
\draw[draw=black,thick] (1,2) to[out=45,in=225] (3,3) to[out=45,in=0] (3,0) to[out=180,in=225] (1,2);
\end{scope}
\draw (1.25,1.75) node {$\Lambda$};
\draw (4.25,3.25) node {$\Lambda^\complement$};
\draw[color=black,step=0.5cm] (0,-1) grid (5,4);
\draw (5.7,3.75) node[above]{$N^{-1}\bb T^d_N$};
\shade[ball color=black](3,1) circle (0.1);
\draw (2.75,0.75) node {$\frac{x}{N}$};
\end{tikzpicture}
\caption{Illustration in dimension two of  $C_N[x,2]$. The sites in $C_N[x,2]$ are those laying in the gray region.} \label{Fig3}
\end{figure}
\noindent Let 
$\Lambda_N =\{x\in \bb T^d_N: \frac{x}{N}\in \Lambda\}$ the set of sites in $\frac{1}{N}\bb T^d_N$ belonging to $\Lambda$. We define now the region $C_N[x,\ell]\subset \bb T^d_N$ by
\begin{equation}\label{counterpart}
C_N[x,\ell]\;:=\; \begin{cases}
B_N[x,\ell] \cap \Lambda_N & \text{if }\frac{x}{N}\in \Lambda\,, \medskip  \\
B_N[x,\ell] \cap \Lambda_N^\complement& \text{if }\frac{x}{N}\in \Lambda^\complement\,, 
\end{cases}
\end{equation}
see Figure~\ref{Fig3} for an illustration. For $\beta\in[1,\infty]$, we define the local density as the average over $C_N[x,\ell]$, that is,
\begin{equation}\label{localmean1}
\eta^{\varepsilon N}(x)\;:=\;\frac{1}{\#C_N[x,\eps N]}\sum_{y\in C_N[x,\eps N]}\eta(y)\,.
\end{equation}

\begin{lemma}\label{r2}
Fix $\beta\in [1,\infty]$. Let $f$ be a density with respect to the invariant measure $\nu_\theta$, let $\lambda_N:\bb T_N^d\to\bb R$ a function such that $\Vert \lambda_N\Vert_\infty\leq M<\infty$ and $\gamma>0$. Then, the following inequalities hold: 
\begin{equation}\label{ineq5.10}
\begin{split}
&\int \gamma N\! \sum_{x\in \Gamma_N}\!\lambda_N(x)\big\{\eta(x)-\eta^{\varepsilon N}(x)\big\}f(\eta)\nu_\theta(d\eta)\;\leq\;\pfrac{1}{2}\gamma^2 M^2 \mc O(N^{d})d\eps+N^2\mf D_N(\sqrt{f})
\end{split}
\end{equation}
 and
\begin{equation}\label{ineq5.11}
\begin{split}
&\int \gamma \sum_{x\in\bb T^d_N}  \lambda_N(x)\{\eta(x)-\eta^{\varepsilon N}(x)\}f(\eta)\nu_\theta(d\eta)\;\leq\; \pfrac{1}{2}\gamma^2 M^2 \mc O(N^{d-1})d\eps+N^2\mf D_N(\sqrt{f})\,.
\end{split}
\end{equation}
\end{lemma}

\begin{proof} 
Let us  prove the inequality  \eqref{ineq5.11}. As commented in the beginning of this subsection, the local average $\eta^{\eps N}$ is taken over  $C_N[x,\eps N]$. Thus, we can write
 \begin{align}
&\int\lambda_N(x)\{\eta(x)-\eta^{\varepsilon N}(x)\}f(\eta)\nu_\theta(d\eta)\nonumber\\
&=\int \lambda_N(x)\Big\{\frac{1}{\#C_N[x,\eps N]}\sum_{y\in C_N[x,\eps N]}\big(\eta(x)-\eta(y)\big)\Big\}f(\eta)\nu_\theta(d\eta)\,.\label{eqtag}
\end{align} 
For each $y\in C[x,\eps N]$, let $\gamma(x,y)$ be a polygonal path of minimal length connecting $x$ to $y$ which does not crosses $\p \Lambda$. 
That is, $\gamma(x,y)$ is a sequence of sites $(a_0,\ldots,a_M)$ such that  $x=a_0$, $y=a_M,  \Vert a_i-a_{i+1}\Vert_1 =1$ and $\xi_{a_,a_{i+1}}=1$ for $i=0,\ldots,M-1$, and $\gamma(x,y)$ has minimal length, that is, $M=M(x,y)=\Vert x-y\Vert_1+1$.
Now we repeat the steps in the proof of Lemma~\ref{r1}, observing that in this case the sum  will be over $\bb T^d_N$, obtaining that \eqref{eqtag} is  bounded from above by 
\begin{align*}
& \frac{1}{2\#C_N[x,\eps N]}\sum_{x\in\bb T^d_N}\sum_{y\in C_N[x,\eps N]}\sum_{\ell=1}^{M(x,y)-1}\Bigg[\frac{1}{2A} \int\Big(\sqrt{f(\eta^{a_\ell, a_{\ell-1}})}-\sqrt{f(\eta)}\Big)^2 \,d\nu_\theta\Bigg.\\
&+\frac{A}{2}\int\Big( \lambda_N(x)\Big)^2(\eta(a_\ell)-\eta(a_{\ell-1}))^2\Big(\sqrt{f(\eta^{a_\ell, a_{\ell-1}})}+\sqrt{f(\eta)}\Big)^2 \, d\nu_\theta\Bigg] \,.
\end{align*}
We can bound the first parcel in the sum above by $\frac{1}{2A}\mf D_N(\sqrt{f})$ and the second parcel  by
\begin{align*} 
\begin{split}
&\frac{1}{2\#C_N[x,\eps N]}\sum_{x\in\bb T^d_N}\sum_{y\in C_N[x,\eps N]}\sum_{\ell=1}^{M(x,y)-1} \frac{4AM^2}{2}\\
& \leq\;\frac{1}{\#C_N[x,\eps N]}\sum_{y\in C_N[x,\eps N]}AM^2\mc O(N^{d})d\varepsilon N\;=\; AM^2\mc O(N^{d})d\varepsilon N\,.
\end{split}
\end{align*} 
We hence have
\begin{align*}
\int \sum_{x\in \bb T^d_N}\lambda_N(x)\big\{\eta(x)-\eta^{\varepsilon N}(x)\big\}f(\eta)\nu_\theta(d\eta)\;\leq\; AM^2 \mc O(N^{d})d\varepsilon N+\frac{1}{2A}\mf D_N(\sqrt{f}) \, .
\end{align*}
Then, multiplying the inequality above by $\gamma$ gives us 
\begin{align*}
\int \gamma \sum_{x\in \bb T^d_N}\lambda_N(x)\big\{\eta(x)-\eta^{\varepsilon N}(x)\big\}f(\eta)\nu_\theta(d\eta)\;\leq\; A\gamma \mc O(N^{d}) M^2 d\varepsilon N+\frac{\gamma}{2A}\mf D_N(\sqrt{f}) \,. 
\end{align*}
Now choosing $A=\gamma N^{-2}/2$ the proof of \eqref{ineq5.10} ends.
The proof of inequality \eqref{ineq5.10} similar to the proof of Lemma~\ref{r1}, under the additional feature that rates of bonds over a path connecting two sites will be always equal to one, which facilitates the argument.
\end{proof}

\begin{lemma}[Replacement lemma] \label{replacement2} 
Fix $\beta\in [1,\infty]$. Let $\lambda_N:\bb T_N^d\to\bb R$  be a sequence of functions such that $\Vert \lambda_N\Vert_\infty\leq c<\infty$. Then,
\begin{equation*}
\varlimsup_{\varepsilon\rightarrow0} \varlimsup_{N\rightarrow\infty} \bb E^\beta_{\mu_N}\Big[\Big|\int_0^t \frac{1}{N^{d-1}}\sum_{x\in\Gamma_N} \lambda_N(x) \{\eta_s^{\varepsilon N}(x)-\eta_s(x)\}\,ds\Big|\Big]\;=\;0
\end{equation*}
and
\begin{equation*}
\varlimsup_{\varepsilon\rightarrow0} \varlimsup_{N\rightarrow\infty}  \bb E^\beta_{\mu_N}\Big[\Big|\int_0^t\frac{1}{N^{d}} \sum_{x\in\bb T^d_N}\lambda_N(x)\{\eta_s^{\varepsilon N}(x)-\eta_s(x)\}\,ds\Big|\Big]\;=\;0 \, .
\end{equation*}

\end{lemma}

\begin{proof}
The proof is similar to the one of Lemma \ref{replacement}, being sufficient to show that expressions
\begin{align*}
\Phi_N^2\;&:=\;\sup_{f\, \textrm{density}} \Big\{\int\gamma N\sum_{x\in\Gamma_N}\lambda_N(x)\{\eta^{\varepsilon N}(x)-\eta(x)\}f(\eta)d\nu_\theta - N^{2}\mf D_N(\sqrt{f})\Big\},\\
\Phi_N^3\;&:=\;\sup_{f\, \textrm{density}} \Big\{\int\gamma\sum_{x\in \bb T^d_N}\lambda_N(x)\{\eta^{\varepsilon N}(x)-\eta(x)\}f(\eta)d\nu_\theta - N^{2}\mf D_N(\sqrt{f})\Big\}
\end{align*}
satisfy
\begin{align*}
&\lim_{N\to\infty}\frac{ t\Phi_N^2}{\gamma N^d}\;=\;0\qquad \text{and} \qquad 
\lim_{N\to\infty}\frac{ t\Phi_N^3}{\gamma N^d}\;=\;0\,,
\end{align*}
which is a    consequence of Lemma \ref{r2}, finishing the proof.
\end{proof}

\subsection{Energy Estimates} 
In this subsection, consider  $\beta\in[1,\infty]$.
Our goal here  is to prove that any limit point $\bb Q^\beta_*$ of the sequence $\{\bb Q^{\beta, N}_{\mu_N}:N>1\}$ is concentrated on trajectories $\rho(t,u) du$ with \textit{finite energy}, meaning that $\rho(t,u)$ belongs to a suitable Sobolev space. 

 This  result  plays a both role in the uniqueness of weak solutions of \eqref{edpbc} and in the characterization of limit points. The fact that $\bb Q^\beta_*$ is concentrated in trajectories with density with respect to the Lebesgue measure of the form $\rho(t,u) du$, with $0\leq \rho\leq 1$, is a consequence of  maximum of one particle per site, see \cite{kl}. The issue here is to prove that the density $\rho(t,u)$ belongs to the Sobolev space $\sobH$, see  Section~\ref{s2} for its definition. 

Assume without loss of generality that the entire sequence $\{\bb Q^{\beta,N}_{\mu_N}: \, N\geq1\}$ weakly converges  to $\bb Q^\beta_*$. 
Let $B[u,\eps]:= \{r\in \bb T^d: \Vert r-u\Vert_\infty<\eps\}$ and
\begin{equation*}
C[u,\eps]\;:=\; \begin{cases}
B[u,\eps] \cap \Lambda & \text{if }u\in \Lambda\,,  \\\medskip
B[u,\eps] \cap \Lambda^\complement& \text{if }u\in \Lambda^\complement\,, 
\end{cases}
\end{equation*}
where we have  written $\Vert \cdot \Vert_\infty$ for the supremum norm on the continuous torus $\bb T^d=[0,1)^d$, that is, $\Vert (u_1,\ldots,u_d)\Vert_\infty = \max\big\{|u_1|\wedge |1-u_1|,\ldots,|u_d|\wedge |1-u_d|\big\}$. See Figure~\ref{Fig4} for an illustration.
\begin{figure}[H]
\centering
\begin{tikzpicture}
\begin{scope}
  \clip (1,2.2) to[out=45,in=225] (3,3.2) to[out=45,in=0] (3,0.2) to[out=180,in=225] (1,2.2);
  \fill[fill=lightgray] (2,0) rectangle (4,2);
\end{scope}
\begin{scope}[yshift=0.2cm]
\draw[draw=black,thick] (1,2) to[out=45,in=225] (3,3) to[out=45,in=0] (3,0) to[out=180,in=225] (1,2);
\end{scope}
\draw (1.25,1.75) node {$\Lambda$};
\draw (4.25,3.25) node {$\Lambda^\complement$};
\draw[color=black] (0,-1) rectangle (5,4);
\draw (5.5,3.75) node[above]{$\bb T^d$};
\shade[ball color=black](3,1) circle (0.08);
\draw (2.75,0.75) node {$u$};
\draw[dashed,thick] (2,0) rectangle (4,2);
\end{tikzpicture}
\caption{Illustration in dimension two of $C[u,\eps]$, which is represented by the  region in gray, while $B[u,\eps]$ is represented by the square delimited by the dashed line. Note that $C[u,\eps]$ is the continuous counterpart of $C_N[x,\ell]$ defined in \eqref{counterpart}.} \label{Fig4}
\end{figure}
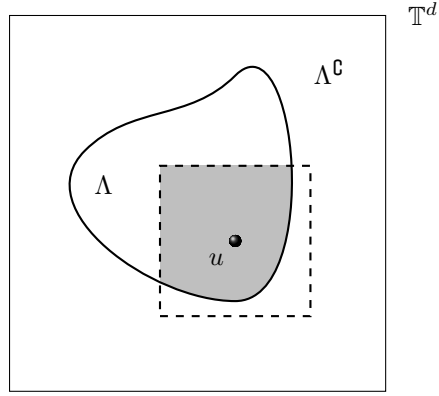
\noindent  We define an approximation of the identity $\iota_\eps$ in the continuous torus $\bb T^d$  by
\begin{equation}\label{approxidentity}
\iota_\eps(u,v)\;:=\; \frac{1}{|C[u,\eps]|} \mathbf{1}_{C[u, \varepsilon]}(v)\,,
    \end{equation}
    where $|C[u,\eps]|$ above denotes the Lebesgue measure of the set $C[u,\eps]$.
Recall that the convolution of a measure $\pi$ with $\iota_\eps$ is defined by
\begin{equation} \label{convolution}
(\pi\ast\iota_\eps)(u)\;=\;\int_{\bb T^d}\iota_\eps(u,v)\pi(dv)\quad \text{ for any }u\in \bb T^d\,.
\end{equation}
Given a function $\rho$, the convolution $\rho\ast\iota_\eps$ shall be understood as the convolution of the measure $\rho(v) dv$ with $\iota_\eps$. 
An important remark now is  the equality
\begin{align}\label{remark}
(\pi^N_t\ast\iota_\eps)\big(\pfrac{x}{N}\big)\;=\;\eta_t^{\varepsilon N}(x)+\mc O\big((\eps N)^{1-d} \big)\,,
\end{align}
where $\eta^{\eps N}_t$ has been defined in \eqref{localmean1}, being the small error above  due to the fact that sites on the boundary of $C_N[x,\ell]$  may or may not belong to $C[u,\eps]$  when taking $u=x/N$ and $\ell=\eps N$. 
Given a  function $H:\bb T^d\rightarrow \bb R$,  let
\begin{equation}\label{eq5.16}
V_N(\varepsilon, j, H,  \eta):=\frac{1}{ N^d}\!\! \sum_{x\in\bb T^d_N}\!\!H\big(\pfrac{x}{N}\big)\frac{\{\eta(x)-\eta(x+\varepsilon N e_j)\}}{\eps}-\frac{2}{N^d}\!\sum_{x\in\bb T^d_N}\!\!\Big(H\big(\pfrac{x}{N}\big)\Big)^2.
\end{equation}
\begin{lemma} \label{lemma5.5}
Consider $H_1, \ldots, H_k$ functions in $C^{0,1}([0,T]\times \bb T^d)$ with compact support contained in $[0,T]\times (\bb T^d\backslash\partial\Lambda)$. Hence, for every $\varepsilon >0$ and $j=1,\ldots,d$,
\begin{equation} \label{5.8e}
\varlimsup_{\delta\rightarrow0} \varlimsup_{N\rightarrow\infty}\bb E_{\mu_N}^{\beta}\Big[\max_{1\leq i \leq k}\Big\{\int_0^T V_N(\varepsilon, j,  H_i(s, \cdot), \eta_s^{\delta N})\,ds\Big\}\Big]\;\leq \;\kappa_0\,,
\end{equation}
where $\kappa_0$ has been defined in \eqref{cte}.
\end{lemma}

\begin{proof} Provided by Lemma~\ref{replacement2}, it is enough to prove that
\begin{equation*} \label{using}
\varlimsup_{N\rightarrow\infty}\,\,\bb E^\beta_{\mu_N}\Big[\max_{1\leq i \leq k}\Big\{\int_0^t V_N(\varepsilon, j,  H_i(s,\cdot), \eta_s)\,ds\Big\}\Big]\;\leq\; \kappa_0\, .
\end{equation*}
By  the entropy inequality, for each fixed $N$, the  expectation above is smaller than
\begin{align*} 
\frac{H(\mu^N|\nu_\theta)}{N^d}+\frac{1}{N^d}\log \bb E_{\nu_\theta}\Big[\exp \Big\{\max_{1\leq i \leq k} N^d\Big\{\int_0^T V_N(\varepsilon, j,  H_i(s, \cdot), \eta_s)\,ds\Big\}\Big\}\Big]\,.
\end{align*}
Using \eqref{cte}, we bound the  first parcel above by  $\kappa_0$. Since $\exp\big\{\max_{1\leq i \leq k} a_j\big\}\leq \sum_{1\leq i \leq k}\exp\{a_j\}$ and by \eqref{A.2}, we conclude that the limsup as $N\uparrow\infty$ of the second parcel above  is less than or equal to 
\begin{align*} 
&\varlimsup_{N\rightarrow\infty} \frac{1}{N^d}\log \bb E_{\nu_\theta}\Big[\sum_{1\leq i \leq k}\exp \Big\{ N^d \int_0^T V_N(\varepsilon, j,  H_i(s, \cdot), \eta_s)\,ds\Big\}\Big]\\
&=\max_{1\leq i \leq k} \varlimsup_{N\rightarrow\infty} \frac{1}{N^d}\log \bb E_{\nu_\theta}\Big[\exp \Big\{ N^d \int_0^T V_N(\varepsilon, j,  H_i(s, \cdot), \eta_s)\,ds\Big\}\Big]\,.
\end{align*}
Thus, in order to conclude the proof, it is enough to show that the limsup above is non positive for each $i=1,\ldots,k$. By the Feynman-Kac formula (see \cite[p. 332, Lemma 7.2]{kl}) for each fixed $N$ and $d\geq 2$, 
\begin{align}
& \frac{1}{N^d}\log \bb E_{\nu_\theta}\Big[\exp \Big\{ N^d \int_0^T V_N(\varepsilon, j,  H_i(s, \cdot), \eta_s)\,ds\Big\}\Big]\label{eq5.18} \\
&\leq \int_0^T \sup_f \Big\{\int V_N(\varepsilon, j,  H_i(s, \cdot), \eta)f(\eta) d \nu_\theta -N^{2-d}\mf D_N(\sqrt{f})\Big\}\,ds\,,\label{eq5.19}
\end{align}
where the  supremum above is taken over all probability densities $f$ with respect to $\nu_\theta$. By assumption, each of the functions $\{H_i : i=1,\ldots,  k\}$ vanishes in a neighborhood of $\partial\Lambda$.
Thus,  we make following observation about the first sum in the RHS of \eqref{eq5.16}: for small $\eps$, non-zero summands are such that  $x/N$ 	and $(x+\eps Ne_j)N$ lay both in $\Lambda$ or both in $\Lambda^\complement$. Henceforth, in such a case, it is possible to find a path no slow bonds connecting $x$ and $x+\eps Ne_j$.
 Keeping this in mind, we can repeat the arguments in the proof of Lemma~\ref{r2} to deduce that 
 \begin{align*}
 & \int \frac{1}{N^d} \sum_{x\in\bb T^d_N}H\big(\pfrac{x}{N}\big)\frac{\{\eta(x)-\eta(x+\varepsilon N e_j)\}}{\eps} f(\eta) d \nu_\theta\\
 &\leq  N^{2-d}\mf D_N(\sqrt{f})+ \frac{2}{N^d}\sum_{x\in\bb T^d_N}\Big(H\big(\pfrac{x}{N}\big)\Big)^2\,.
 \end{align*}
Plugging this inequality into \eqref{eq5.19}  implies that \eqref{eq5.18} has a nonpositive limsup, showing \eqref{using} and therefore finishing the proof.
\end{proof}

\begin{lemma} \label{5.7}
\begin{equation*}
\bb E_{\bb Q^\beta_*}\left[\sup_{H}\left\{\int_0^T\!\!\int_{\bb T^d}(\partial_{u_j}H)(s,u)\rho(s,u)duds-2\int_0^T\!\!\int_{\bb T^d}\left(H(s,u)\right)^2duds\right\}\right]\;\leq\; \kappa_0\,,
\end{equation*}
where the supremum is carried over all functions $H\in C^{0,1}([0,T]\times\bb T^d)$ with compact support contained in $[0,T]\times (\bb T^d\backslash \partial\Lambda)$.
\end{lemma}

\begin{proof}
Consider a sequence $\{H_i:\,i\geq1\}$ dense in the subset of $C^2([0,t]\times\bb T^d)$ of functions with support contained in $[0,T]\times(\bb T^d\backslash \p\Lambda)$, being the density with respect  to the norm $\Vert H\Vert_\infty+\Vert\p_u H\Vert_\infty$.
Recall  we are assuming  that $\{\bb Q_{\mu_N}^{\beta, N}: N\geq 1\}$ converges to $\bb Q^\beta_*$. Then, by \eqref{5.8e} and the Portmanteau Theorem,
\begin{align*}
\varlimsup_{\delta\rightarrow0}\bb E_{\bb Q_{*}^{\beta}}&\Big[\max_{1\leq i\leq k}\Big\{\frac{1}{\varepsilon}\int_0^T\int_{\bb T^d}H_i(s,u) )\{\rho_s^{\delta}(u)-\rho_s^{\delta}(u+\varepsilon e_j)\}\,duds\\
&-2\int_0^T\int_{\bb T^d}(H_i(s,u))^2\,duds\Big\}\Big]\;\leq\; \kappa_0,
\end{align*}
where $\rho_s^\delta(u)=(\rho_s\ast \iota_\delta)(u)$ as defined in \eqref{convolution}. Letting $\delta\downarrow0$, the Lebesgue Differentiation Theorem assures that $\rho_s^\delta(u)$ converges almost surely to $\rho_s$. Then, performing a change of variables and  letting $\varepsilon\downarrow0$, we obtain that 
\begin{equation*}
\bb E_{\bb Q_{*}^{\beta}}\Big[\max_{1\leq i\leq k}\Big\{\int_0^T\!\!\int_{\bb T^d} (\p_{u_j}H_i(s,u)) \rho_s(u)\,duds-2\int_0^T\!\!\int_{\bb T^d}(H_i(s,u))^2\,duds\Big\}\Big]\leq \kappa_0.
\end{equation*}
Since the maximum increases to the supremum, we conclude the lemma by applying the Monotone Convergence Theorem to $\{H_i:\, i\geq1\}$, which is a dense sequence in the subset of functions $C^2([0,T]\times\bb T^d)$ with compact support contained in $[0,T]\times \bb (T^d \backslash \p\Lambda)$. 
\end{proof} 
\begin{proposition}\label{Prop5.7}
The measure $\bb Q^\beta_*$ is concentrated on paths $\pi(t,u)=\rho(t,u)du$ such that $\rho\in \sobH$.
\end{proposition}

\begin{proof} 
Denote by $\ell:C^2([0,T]\times\bb T^d)\rightarrow \bb R$ the linear functional defined by 
$$\ell(H)\;=\;\int_0^T\int_{\bb T^d} (\partial_{u_j}H)(s,u)\rho(s,u)\, du\,ds\,.$$
Since the set of functions $H\in C^2([0,T]\times\bb T^d)$ with support contained in $[0,T]\times(\bb T^d\backslash\partial\Lambda)$ is dense in $L^2([0,T]\times\bb T^d)$ and since by Lemma \ref{5.7} $\ell$ is a $\bb Q^\beta_*$-a.s. bounded functional in $C^2([0,T]\times\bb T^d)$, we can extend it to a $\bb Q^\beta_*$-a.s. bounded functional in $L^2([0,T]\times\bb T^d)$, which is a Hilbert space. Then, by the Riesz Representation Theorem, there exists a function $G\in L^2([0,T]\times\bb T^d)$ such that
$$\ell(H)\;=\;-\int_0^T\int_{\bb T^d} H(s,u)G(s,u)\, du\,ds\,,$$
concluding  the proof.
\end{proof}

\section{Characterization of limit points} \label{s6}
Before going into the details of each regime $\beta\in[0,1)$, $\beta=1$ or $\beta\in(1,\infty]$, we make some useful considerations for all cases.

We will prove in this section that all limit points of the sequence $\{\bb Q^{\beta,N}_{\mu_N}: \, N\geq1\}$ are concentrated on trajectories of measures $\pi(t,du)=\rho(t,u)\,du$, whose density $\rho(t,u)$ with respect to the Lebesgue measure is the weak solution of the hydrodynamic equation \eqref{edpheat}, \eqref{edp12} or \eqref{edpbc} for each corresponding value of $\beta$.
Provided by tightness, let $\bb Q_*^\beta$ be a limit point of the sequence $\{\bb Q^{\beta,N}_{\mu_N}: \, N\geq1\}$ and assume, without loss of generality, that $\{\bb Q^{\beta,N}_{\mu_N}: \, N\geq1\}$ converges to $\bb Q^\beta_*$. 

Since there is at most one particle per site, it is easy to show that $\bb Q_*^\beta$ is concentrated on trajectories $\pi(t,du)$ which are absolutely continuous with respect to the Lebesgue measure $\pi(t,du)=\rho(t,u)\,du$ and whose density $\rho(t,\cdot)$,  is nonnegative and bounded by one.  Recall the martingale $M_t^N(H)$  in~\eqref{M}. 
\begin{lemma} \label{bm}
If 
\begin{enumerate}[a)]
\item $\beta\in[0,1)$ and $H\in C^2(\bb T^d)$, or
\item $\beta\in [1,\infty]$ and $H\in C^2(\bb T^d \backslash \p\Lambda)$, 
\end{enumerate}
then, for all $\delta>0$, 
\begin{equation}\label{quamarprop}
\lim_{N\rightarrow\infty} \bb P^N_{\mu_N}\Big[\sup_{0\leq t\leq T}|M_t^N(H)|>\delta\Big]\;=\;0\, .
\end{equation}
\end{lemma}
\begin{proof}
Item a) has been already proved in \eqref{limmart}. For item b), recalling \eqref{quamar0} note that
\begin{equation}\label{61}
\<M^{N}(H)\>_t \;\leq\;  \frac{T}{N^{2d-2}}\sum_{j=1}^d\sum_{x\in\bb T_{N}^d} \xi^{N}_{x,x+e_j}\Big[H(\pfrac{x+e_j}{N})-H(\pfrac{x}{N})\Big]^2\, . 
\end{equation}
Since $H\in C^2(\bb T^d \backslash \p\Lambda)$, $H$ is differentiable with bounded derivative except over $\p\Lambda$. Therefore, if the edge  $x, \, x+e_j$ is not a slow bond, then
\begin{equation}\label{62}
 \xi^{N}_{x,x+e_j}\Big[H(\pfrac{x+e_j}{N})-H(\pfrac{x}{N})\Big]^2 \;\leq\; \frac{1}{N^2}\Vert\p_{u_j}H\Vert^2_\infty \, . 
\end{equation}
On the other hand, if the edge  $x, \, x+e_j$ is  a slow bond, then
\begin{align}\label{63}
 \xi^{N}_{x,x+e_j}\Big[H(\pfrac{x+e_j}{N})-H_t(\pfrac{x}{N})\Big]^2 &\;\leq\; \frac{4\alpha \Vert H\Vert_{\infty}^2}{N^{\beta}} \,.
\end{align}
Since the number of slow bonds is of order $\mc O(N^{d-1})$, plugging \eqref{62} and \eqref{63} into \eqref{61} gives us $\<M^{N}(H_t)\>_t\leq \mc O(1/N^d)$.
' Then, Doob's inequality concludes the proof.
\end{proof}

\subsection{Characterization of limit points for \texorpdfstring{$\beta\in[0,1)$}{beta em [0,1)}.} \label{6.1}

\begin{proposition} \label{6.1.1}
Let $H\in C^2(\bb T^d)$. Then, for any $\delta>0$, 
\begin{equation*}
\bb Q^\beta_*\Big[\pi.:\, \sup_{0\leq t\leq T}\Big|\<\pi_t, H\> \,-\, \<\pi_0, H \> \,-\, \int_0^t  \, \<\pi_s ,\Delta H \> \,ds\Big| >\delta\Big]\;=\;0\, . 
\end{equation*}
\end{proposition}
\begin{proof}
Since $\bb Q^{\beta, N}_{\mu_N}$ converges weakly to $Q^\beta_*$, by Portmanteau's Theorem (see \cite[Theorem 2.1]{bili}), 
\begin{align}
&\bb Q^\beta_*\Big[\pi.:\, \sup_{0\leq t\leq T}\Big|\<\pi_t, H\> \,-\, \<\pi_0, H \> \,-\, \int_0^t  \, \<\pi_s ,\Delta H \> \,ds\Big| >\delta\Big] \nonumber\\
&\leq\varlimsup_{N\rightarrow\infty}\bb Q^{\beta, N}_{\mu_N}\Big[\pi.:\, \sup_{0\leq t\leq T}\Big|\<\pi_t, H\> \,-\, \<\pi_0, H \> \,-\, \int_0^t  \, \<\pi_s ,\Delta H \> \,ds\Big| >\delta\Big]\label{eq62}
\end{align}
since the supremum above is a continuous function in the Skorohod metric, see Proposition \ref{A.3}. Recall that $\bb Q^{\beta, N}_{\mu_N}$ is the probability measure induced by $\bb P^{\beta}_{\mu_N}$ via the empirical measure. With this in mind and then adding and subtracting $\<\pi_s^N, N^2 \bb L_N H\>$,  expression \eqref{eq62} can be bounded from above by
\begin{equation*}
\begin{split}
&\varlimsup_{N\rightarrow\infty}\bb P^{\beta}_{\mu_N}\Big[\pi.:\, \sup_{0\leq t\leq T}\Big|\<\pi^N_t, H\> \,-\, \<\pi^N_0, H \> \,-\, \int_0^t  \, \<\pi^N_s ,N^2\bb L_N H \> \,ds\Big| >\delta/2\Big]\\
&+\varlimsup_{N\rightarrow\infty}\bb P^{\beta}_{\mu_N}\Big[\pi.:\, \sup_{0\leq t\leq T}\Big|\int_0^t  \, \<\pi^N_s ,\Delta H-N^2\bb L_N H \> \,ds\Big| >\delta/2\Big]\,.
\end{split}
\end{equation*}
By Lemma \ref{bm}, the first term above is null. Since there is at most one particle per site, the second term in last expression is bounded by
\begin{equation*}
\begin{split}
&\varlimsup_{N\rightarrow\infty}\bb P^{\beta}_{\mu_N}\Big[\frac{T}{N^d}\sum_{x\notin\Gamma_N}\Big|\Delta H\Big(\frac{x}{N}\Big)-N^2\bb L_N\Big(\frac{x}{N}\Big)\Big| >\delta/4\Big]\\
&+\varlimsup_{N\rightarrow\infty}\bb P^{\beta}_{\mu_N}\Big[\sup_{0\leq t \leq T}\Big|\int_0^t \frac{1}{N^d}\sum_{x\in\Gamma_N}\Big\{\Delta H\Big(\frac{x}{N}\Big)-N^2\bb L_N\Big(\frac{x}{N}\Big)\Big\}\eta_s(x) \, ds\Big| >\delta/4\Big] \,.
\end{split}
\end{equation*}
Outside $\Gamma_N$, the operator $N^2\bb L_N$ coincides with the discrete Laplacian. Since $H\in C^2(\bb T^d)$,  the first probability above vanishes for $N$ sufficiently large.  Recall that the number  of  elements in $\Gamma_N$ is of order $N^{d-1}$. Applying the triangular inequality, the second expression in the previous sum becomes bounded by the sum of 
\begin{equation}\label{eq6.2}
\varlimsup_{N\rightarrow\infty}\bb P^{\beta}_{\mu_N}\Big[\mc O(N^{-1})T\|\Delta H\|_{\infty}>\delta/8\Big]
\end{equation}
and
\begin{equation}\label{eq6.3}
\varlimsup_{N\rightarrow\infty}\bb P^{\beta}_{\mu_N}\Big[\sup_{0\leq t \leq T}\Big| \int_0^t \frac{1}{N^{d-1}}\sum_{x\in\Gamma_N}N\bb L_N \Big(\frac{x}{N}\Big)\eta_s(x) \, ds\Big| >\delta/8\Big]\,. 
\end{equation}
For large $N$, the probability in \eqref{eq6.2} vanishes. We deal now with \eqref{eq6.3}.
Let $x\in \Gamma_N$.   By definition of $\Gamma_N$, some adjacent bond to $x$ is a \textit{slow bond}. Thus, the opposite vertex to $x$ with respect to this bond is also in $\Gamma_N$, see Figure~\ref{Fig5}.
\begin{figure}[htb!]
\centering
\begin{tikzpicture}
\begin{scope}[yshift=0.2cm]
\draw[draw=black,thick] (1,2) to[out=45,in=225] (3,3) to[out=45,in=0] (3,0) to[out=180,in=225] (1,2);
\end{scope}
\draw (1.25,1.75) node {$\Lambda$};
\draw (4.25,3.25) node {$\Lambda^\complement$};
\draw[color=black,step=1cm] (0,-1) grid (5,4);
\draw (5.7,3.75) node[above]{$N^{-1}\bb T^d_N$};
\shade[ball color=black](3,1) circle (0.1);
\draw (2.75,0.7) node {$\frac{x}{N}$};

\shade[ball color=black](4,1) circle (0.1);
\draw (3.8,0.7) node {$\frac{y}{N}$};

\shade[ball color=black](3,0) circle (0.1);
\draw (2.75,-0.3) node {$\frac{z}{N}$};
\end{tikzpicture}
\caption{Illustration of sites $x,y,z\in \Gamma_N$. We note that two adjacent edges to $x$ are slow bonds, and two adjacent edges are not. Besides,  any opposite vertex to $x$ will be of the form $x\pm e_j$.} \label{Fig5}
\end{figure}
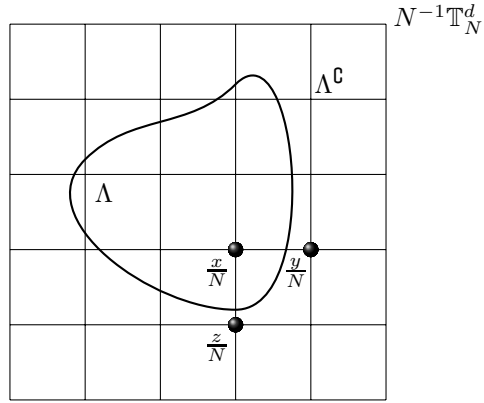

Recall the definition of $\bb L_N$ in \eqref{bbLN}. Whenever $\{x,x-e_j\}$ neither $\{x,x+e_j\}$ are slow bonds, the expression 
\begin{align*}
\xi^N_{x,x+e_j} \, \Big[ H\big(\pfrac{x+e_j}{N}\big) 
- H\big(\pfrac{x}{N}\big) \Big] + \xi^N_{x,x-e_j} \, \Big[H\big(\pfrac{x-e_j}{N}\big) - H\big(\pfrac{x}{N}\big) \Big] 
\end{align*}
is of order $\mc O(N^{-2})$ due to assumption $H\in C^2(\bb T^d)$. Therefore, in \eqref{eq6.3} we can disregard terms of this kind, reducing the proof that \eqref{eq6.3} is null to prove that
\begin{equation}\label{eq6.5}
\varlimsup_{N\rightarrow\infty}\bb P^{\beta}_{\mu_N}\Big[\sup_{0\leq t \leq T}\Big| \int_0^t \frac{1}{N^{d-1}}\hspace{-0.3cm}\sum_{\topo{e=\{x,x+e_j\}}{e \text{ is a slow bond}}}\!\!{\bf A}(e) \, ds\Big| >\delta/16\Big]\;=\;0\,,
\end{equation}
where
\begin{align*}
{\bf A}(e)\;=\; &  \Bigg[\alpha N^{1-\beta}\Big(H\big(\pfrac{x+e_j}{N}\big)-H\big(\pfrac{x}{N}\big)\Big)+\frac{H\big(\pfrac{x-e_j}{N}\big)-H\big(\pfrac{x}{N}\big)}{1/N}\Bigg]\eta_s(x)\\
+\;&  \Bigg[\frac{H\big(\pfrac{x+2e_j}{N}\big)-H\big(\pfrac{x+e_j}{N}\big)}{1/N}+\alpha N^{1-\beta}\Big(H\big(\pfrac{x}{N}\big)-H\big(\pfrac{x+e_j}{N}\big)\Big)\Bigg]\eta_s(x+e_j)\,.
\end{align*}
Since $H$ is smooth, the terms inside  parenthesis involving $N^{1-\beta}$ are of order $\mc O(N^{-\beta})$ and hence negligible. On the other hand, the remaining terms are close  to plus or minus the derivative of $H$ at $x/N$.  We have thus reduced  the proof of \eqref{eq6.5} to the proof of 
\begin{equation}\label{eq6.6}
\varlimsup_{N\rightarrow\infty}\bb P^{\beta}_{\mu_N}\Big[\sup_{0\leq t \leq T}\Big| \int_0^t \frac{1}{N^{d-1}}\hspace{-0.3cm}\sum_{\topo{e=\{x,x+e_j\}}{e \text{ is a slow bond}}}\!\!\p_{u_j}H\big(\pfrac{x}{N}\big) \big(\eta_s(x+e_j)-\eta_s(x)\big) \, ds\Big| >\delta/32\Big]\;=\;0\,.
\end{equation}
Let $t_0=0<t_1<\cdots<t_n=T$ be a partition of $[0,T]$ with mesh bounded by an arbitrary
$\tilde{\eps}>0$. Via the triangular inequality, if we prove that 
\begin{equation*}
\sum_{k=0}^n\varlimsup_{N\to\infty} \bb P^{\beta}_{\mu_N} \Big[ \,\Big\vert
 \int_0^{t_k} \frac{1}{N^{d-1}}\hspace{-0.3cm}\sum_{\topo{e=\{x,x+e_j\}}{e \text{ is a slow bond}}}\!\!\p_{u_j}H\big(\pfrac{x}{N}\big) \big(\eta_s(x+e_j)-\eta_s(x)\big) \, ds\,\Big\vert
 \, > \, \delta \,\Big]
\end{equation*}
vanishes, then we will conclude that \eqref{eq6.6} vanishes as well.
Therefore, it is enough now to show that, for any $\delta>0$ and any $t\in[0,T]$,
\begin{equation*}
\varlimsup_{N\to\infty} \bb P^{\beta}_{\mu_N} \Big[ \,\Big\vert
 \int_0^{t}
\frac{1}{N^{d-1}}\hspace{-0.3cm}\sum_{\topo{e=\{x,x+e_j\}}{e \text{ is a slow bond}}}\!\!\p_{u_j}H\big(\pfrac{x}{N}\big) \big(\eta_s(x+e_j)-\eta_s(x)\big) \, ds\,\Big\vert
 \, > \, \delta \,\Big]\;=\;0\,.
\end{equation*}
 Markov's inequality then allows us to bound the expression above  by
\begin{equation} \label{markovex}
\varlimsup_{N\to\infty} \delta^{-1}\bb E^\beta_{\mu_N}\Big[\, \Big| \int_0^{t}
\frac{1}{N^{d-1}}\hspace{-0.3cm}\sum_{\topo{e=\{x,x+e_j\}}{e \text{ is a slow bond}}}\!\!\p_{u_j}H\big(\pfrac{x}{N}\big) \big(\eta_s(x+e_j)-\eta_s(x)\big) \, ds\,\Big|\,\Big]\,.
\end{equation}
Adding and subtracting $\eta^{\eps N}_s(x)$ and $\eta^{\eps N}_s(x+e_j)$, we bound \eqref{markovex} from above by
\begin{equation*}
\begin{split}
&\varlimsup_{N\to\infty} \delta^{-1} \bb E^\beta_{\mu_N}\Big[\, \Big| \int_0^{t}
\frac{1}{N^{d-1}}\hspace{-0.3cm}\sum_{\topo{e=\{x,x+e_j\}}{e \text{ is a slow bond}}}\!\!\p_{u_j}H\big(\pfrac{x}{N}\big) \big(\eta_s(x+e_j)-\eta_s^{\eps N}(x+e_j)\big) \, ds\,\Big|\,\Big]\\
&+\varlimsup_{N\to\infty} \delta^{-1}\bb E^\beta_{\mu_N}\Big[\, \Big| \int_0^{t}
\frac{1}{N^{d-1}}\hspace{-0.3cm}\sum_{\topo{e=\{x,x+e_j\}}{e \text{ is a slow bond}}}\!\!\p_{u_j}H\big(\pfrac{x}{N}\big) \big(\eta^{\eps N}_s(x+e_j)-\eta_s^{\eps N}(x)\big) \, ds\,\Big|\,\Big]\\
&+\varlimsup_{N\to\infty} \delta^{-1} \bb E^\beta_{\mu_N}\Big[\, \Big| \int_0^{t}
\frac{1}{N^{d-1}}\hspace{-0.3cm}\sum_{\topo{e=\{x,x+e_j\}}{e \text{ is a slow bond}}}\!\!\p_{u_j}H\big(\pfrac{x}{N}\big) \big(\eta^{\eps N}_s(x)-\eta_s(x)\big) \, ds\,\Big|\,\Big] \,.
\end{split}
\end{equation*}
Since $|\{\eta_s^{\varepsilon N}(x+e_j)-\eta_s^{\varepsilon N}(x)\}|\leq \frac{2(\varepsilon N)^{d-1}}{(\varepsilon N)^d}=\frac{2}{\varepsilon N}$, $|\Gamma_N|$ is of order $N^{d-1}$ and  $\Vert\p_{u_j}H\Vert_\infty<\infty$, the second term above vanishes. For the remaining terms, we  apply Lemma~\ref{replacement}, finishing the proof. 
\end{proof}

\subsection{Characterization of limit points for \texorpdfstring{$\beta=1$}{beta=1}.} \label{6.2}
This subsection is devoted to the proof of the next proposition. Keep in mind that Proposition~\ref{Prop5.7} allows us to write $\pi(t,u)=\rho(t,u)du$ when considering the measure $\bb Q^\beta_*$.
\begin{proposition}\label{prop62one}
Let $H\in C^2(\bb T^d\backslash \partial\Lambda)$. For all $\delta>0$,
\begin{equation} \label{6.4}
\begin{split}
&\bb Q^\beta_*\Big[\pi.: \sup_{0\leq t\leq T}\Big|\< \rho_t, H\> - \< \rho_0 , H\> - \int_0^t \< \rho_s , \Delta H \> \, ds\\
&-\int_0^t\int_{\p\Lambda}\rho_s(u^+)\sum_{j=1}^d\p_{u_j} H(u^+)\<\vec{\zeta}(u),e_j\>\,dS(u)ds\\
&+\int_0^t \int_{\p\Lambda}\rho_s(u^-)\sum_{j=1}^d\p_{u_j} H(u^-)\<\vec{\zeta}(u),e_j\>\,dS(u)ds\\
&+\!\int_0^t\!\! \int_{\p\Lambda}\!\!\!\alpha (\rho_s(u^-)-\rho_s(u^+))(H(u^+)-H(u^-))\sum_{j=1}^d|\<\vec{\zeta}(u), e_j\>|\,dS(u)ds\Big| >\delta\,\Big]=0.
\end{split}
\end{equation}
\end{proposition}

\begin{figure}[htb!]
\centering
\begin{tikzpicture}
\begin{scope}
\begin{scope}[yshift=0.2cm]
\draw[draw=black,thick] (1,2) to[out=45,in=225] (3,3) to[out=45,in=0] (3,0) to[out=180,in=225] (1,2);
\end{scope}
\draw (2.25,1.75) node {$\Lambda$};
\draw (4.25,3.25) node {$\Lambda^\complement$};
\draw[color=black,step=0.5cm] (0,-1) grid (5,4);
\draw (4.5,4) node[above]{$N^{-1}\bb T^d_N$};
\shade[ball color=black](3,0.5) circle (0.1);
\shade[ball color=black](3.5,0.5) circle (0.1);
\shade[ball color=black](3.5,1) circle (0.1);
\shade[ball color=black](3.5,1.5) circle (0.1);
\shade[ball color=black](3.5,2) circle (0.1);
\shade[ball color=black](3.5,2.5) circle (0.1);
\shade[ball color=black](3.5,3) circle (0.1);
\shade[ball color=black](3,3) circle (0.1);
\shade[ball color=black](2.5,2.5) circle (0.1);
\shade[ball color=black](2,2.5) circle (0.1);
\shade[ball color=black](1.5,2.5) circle (0.1);
\shade[ball color=black](1,2) circle (0.1);
\shade[ball color=black](1,1.5) circle (0.1);
\shade[ball color=black](1.5,1) circle (0.1);
\shade[ball color=black](2,0.5) circle (0.1);
\shade[ball color=black](2.5,0.5) circle (0.1);
\end{scope}

\begin{scope}[xshift=6cm]
\begin{scope}[yshift=0.2cm]
\draw[draw=black,thick] (1,2) to[out=45,in=225] (3,3) to[out=45,in=0] (3,0) to[out=180,in=225] (1,2);
\end{scope}
\draw (2.25,1.75) node {$\Lambda$};
\draw (4.25,3.25) node {$\Lambda^\complement$};
\draw[color=black,step=0.5cm] (0,-1) grid (5,4);
\draw (4.5,4) node[above]{$N^{-1}\bb T^d_N$};
\shade[ball color=lightgray](3,0.5) circle (0.1);
\shade[ball color=lightgray](3.5,0.5) circle (0.1);
\shade[ball color=black](3.5,3) circle (0.1);
\shade[ball color=black](3,3) circle (0.1);
\shade[ball color=black](2.5,2.5) circle (0.1);
\shade[ball color=black](2,2.5) circle (0.1);
\shade[ball color=black](1.5,2.5) circle (0.1);
\shade[ball color=black](1,2) circle (0.1);
\shade[ball color=lightgray](1,1.5) circle (0.1);
\shade[ball color=lightgray](1.5,1) circle (0.1);
\shade[ball color=lightgray](2,0.5) circle (0.1);
\shade[ball color=lightgray](2.5,0.5) circle (0.1);
\begin{scope}[xshift=-0.25cm,yshift=-0.25cm]
\draw[very thick,->] (0,-1) -- (0,0) node[above]{$e_2$};
\draw[very thick,->] (0,-1) -- (1,-1) node[right]{$e_1$};
\end{scope}
\end{scope}
\end{tikzpicture}
\caption{In the left, an illustration of the set $\Gmenos$, whose elements are represented by black balls. In the right, an illustration of the sets $\GmenosjL$ and $\GmenosjR$ for $j=2$, whose elements are represented by gray and black balls, respectively.} \label{Fig6}
\end{figure}
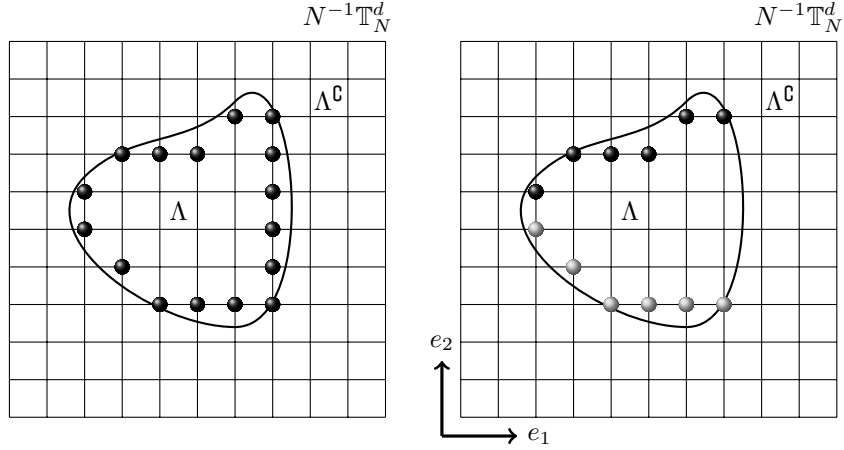
Let us gather some ingredients for the proof of above.  The first one is a suitable expression for $N\bb L_N $ over $\Gamma_N$. Define
\begin{equation}
\begin{split}
\Gmenos & \;=\; \Gamma_N \cap \big\{x\in\bb T^d_N\,: \, \pfrac{x}{N}\in \Lambda \big\}\quad \text{ and }\\
\Gmais & \; =\;\Gamma_N \cap \big\{x\in\bb T^d_N\,: \, \pfrac{x}{N}\in \Lambda^\complement \big\}
\end{split}
\end{equation}
Such a  notation has been chosen to agree with \eqref{maismenos}. Let us focus on $\Gmenos$, being the analysis for $\Gmais$ completely analogous. 
 It is convenient to consider the decomposition  $\Gmenos=\bigcup_{j=1}^d \Gmenosj$, where 
 \begin{align*}
 &\Gmenosj \;=\; \GmenosjL \cup \GmenosjR\,, \qquad \text{with}
 \end{align*}
  \begin{align*}
 &\GmenosjL \;=\;\Big\{x\in\Gmenos\,: \,  \frac{x-e_j}{N}\in \Lambda^\complement \Big\}\;\text{ and }\;
 \GmenosjR \;=\; \Big\{x\in\Gmenos\,: \,  \frac{x+e_j}{N}\in \Lambda^\complement \Big\}\,,
\end{align*}
see Figure~\ref{Fig6} for an illustration. Note that $\GmenosjR$ and $\GmenosjL$ are not  necessarily disjoint for a fixed $j$. Nevertheless, due to the smoothness of $\p\Lambda$, the number of elements in the intersection of these two sets is of order $\mc O(N^{d-2})$, hence negligible to our purposes. We will henceforth assume that $\GmenosjR$ and $\GmenosjL$ are disjoint sets for all $j=1,\ldots, d$. 
\begin{remark}\rm
At first sight, the reader may imagine that $\Gmenos$ is equal to $\GmenosjL\cup \GmenosjR$ for any $j$, or at least very to close to. This is false, as illustrated by Figure~\ref{Fig6}. Moreover, for $i\neq j$ and large $N$, the sets  $\Gmenosj$ and $\Gmenosi$ in general are not disjoint with a no negligible intersection.
\end{remark}

Define now $$N\bb L^j_N H(\pfrac{x}{N})=
N\xi^N_{x,x+e_j}\big(H(\pfrac{x+e_j}{N})-H(\pfrac{x}{N})\big)+N\xi^N_{x,x-e_j}\big(H(\pfrac{x-e_j}{N})-H(\pfrac{x}{N})\big)\,.$$
Then, by By Fubini's Lemma,
\begin{align} 
&\sum_{x\in\Gmenos}N\bb L_N H\big(\pfrac{x}{N}\big)\eta^{\varepsilon N}_s(x)\;=\; \sum_{x\in\Gmenos}\sum_{j=1}^{d}N\bb L^j_N H\big(\pfrac{x}{N}\big)\eta^{\varepsilon N}_s(x)\nonumber \\
&=\sum_{j=1}^{d}\Big\{ \sum_{x\in\GmenosjR}N\bb L^j_N H\big(\pfrac{x}{N}\big)\eta^{\varepsilon N}_s(x)+\sum_{x\in\GmenosjL}N\bb L^j_N H\big(\pfrac{x}{N}\big)\eta^{\varepsilon N}_s(x)\Big\}\,.\label{sumnj}
\end{align}
If $x\in \GmenosjR$, then  $\xi^N_{x,x+e_j}=\alpha/N$ and $\xi^N_{x,x-e_j}=1$, see Figure~\ref{Fig5}. In this case, 
\begin{equation*}
N\bb L^j_N H\big(\pfrac{x}{N}\big)\;=\;\alpha \Big(H\big(\pfrac{x+e_j}{N}\big)-H\big(\pfrac{x}{N}\big)\Big)-\p_{u_j}H\big(\pfrac{x}{N}\big)+\mc O(N^{-1})\,.
\end{equation*}
On the other hand, if $x\in \GmenosjL$, then  $\xi^N_{x,x-e_j}=\alpha/N$ and $\xi^N_{x,x+e_j}=1$. In this case, 
\begin{equation*}
N\bb L^j_N H\big(\pfrac{x}{N}\big)\;=\;\p_{u_j}H\big(\pfrac{x}{N}\big)+\alpha\Big(H\big(\pfrac{x-e_j}{N}\big)-H\big(\pfrac{x}{N}\big)\Big)+\mc O(N^{-1})\,.
\end{equation*}
Now, let $\uu:\bb T^d\to\p\Lambda$ be a  function such that 
\begin{equation}\label{uu}
\Vert \uu(u) -u\Vert\;=\; \min_{v\in\p\Lambda} \Vert v -u\Vert\,,
\end{equation}
and $\uu$ is continuous in a  neighborhood of $\p\Lambda$. 
That is, $\uu$ maps $u\in \bb T^d$ to some of its closest points over $\p\Lambda$ and $\uu$ is continuous on the set $(\p\Lambda)^\eps=\{u\in \bb T^d: \text{dist}(u,\p\Lambda)< \eps\}$ for some small $\eps>0$. There are more than  one  function fulfilling \eqref{uu}, but any choice among them will be satisfactory for our purposes, once this function  is continuous near $\p\Lambda$.  With this mind we can  rewrite \eqref{sumnj}, achieving the formula
\begin{equation}\label{sumnj2}
\begin{split}
&\frac{1}{N^{d-1}}\sum_{x\in\Gmenos}N\bb L_N H\big(\pfrac{x}{N}\big)\eta^{\varepsilon N}_s(x)\\
 & \;=\;
\frac{1}{N^{d-1}}\sum_{j=1}^d \Bigg\{\sum_{x\in\GmenosjR}\Big[\alpha \big(H(\uu^+)-H(\uu^-)\big)-\p_{u_j}H(\uu^-)\Big]\eta^{\varepsilon N}_s(x)\\
&\hspace{2.5cm}+\sum_{x\in\GmenosjL}\Big[\p_{u_j}H(\uu^-)+\alpha\big(H(\uu^+)-H(\uu^-)\big)\Big]\eta^{\varepsilon N}_s(x)\Bigg\}\,.
\end{split}
\end{equation}
plus a negligible error, where by $H(\uu^-)$ and $H(\uu^+)$ are the sided limits of $H$ at $\uu$. The dependence of $\uu$  on $x/N$ will be dropped to not overload notation.  
Defining 
 \begin{align*}
 &\Gmaisj \;=\; \GmaisjL \cup \GmaisjR\,, \qquad \text{with}
 \end{align*}
\begin{align*}
 \GmaisjL =\Big\{x\in\Gmais\,: \,  \frac{x+e_j}{N}\in \Lambda \Big\}\;\text{ and }\;
 \GmaisjR = \Big\{x\in\Gmais\,: \,  \frac{x-e_j}{N}\in \Lambda \Big\}\,\,,
\end{align*}
we similarly  have 
\begin{equation}\label{sumnj3}
\begin{split}
&\frac{1}{N^{d-1}}\sum_{x\in\Gmais}N\bb L_N H\big(\pfrac{x}{N}\big)\eta^{\varepsilon N}_s(x)\\
 & \;=\;
\frac{1}{N^{d-1}}\sum_{j=1}^d \Bigg\{\sum_{x\in\GmaisjR}\Big[\p_{u_j}H(\uu^+)+\alpha \big(H(\uu^-)-H(\uu^+)\big)\Big]\eta^{\varepsilon N}_s(x)\\
&\hspace{2.5cm}+\sum_{x\in\GmaisjL}\Big[\alpha\big(H(\uu^-)-H(\uu^+)\big)-\p_{u_j}H(\uu^+)\Big]\eta^{\varepsilon N}_s(x)\Bigg\}\,.
\end{split}
\end{equation}

The second ingredient is about convergence of sums over $\Gamma_N$ towards   integrals over $\p\Lambda$. Let us review some standard facts about integrals over  surfaces.
Consider a smooth compact manifold $\mc M\subset \bb R^d$ of dimension $(d-1)$. Assume that $\mc M$ is the graph of a function $f: R\subset \bb R^{d-1}\to \bb R$, that is, $\mc M=\{(x,f(x)): x\in R\}$. Then, given a smooth function $g:\mc M\to \bb R$, the surface integral of $g$ over $\mc M$ will be given by
\begin{equation}\label{eqGdS}
\begin{split}
&\int_{\mc M}g(u)\, dS(u) \;=\; \int_{R}g(x,f(x)) \frac{dx}{|\cos (\gamma(x,f(x)))|}\\
& =\; \int_{R}g\big(x_1,\ldots,x_{d-1},f(x_1,\ldots,x_{d-1})\big) \frac{dx_1\cdots dx_{d-1}}{|\<\vec{\zeta}(x_1,\ldots,x_{d-1}),e_d\>|}\,,
\end{split}
\end{equation}
where $\gamma(x,f(x))$ is defined as the angle between the normal exterior vector $\vec{\zeta}(u)=\vec{\zeta}(x_1,\ldots,x_{d-1})$ and $e_d$, the $d$-th element of the canonical basis of $\bb R^d$. Of course,  a manifold in general is only locally a graph of a function as above. Nevertheless, the notion of partition of unity allows to use this local property to evaluate a surface integral. Recall the definition of $\uu$ given in \eqref{uu}.
\begin{lemma} \label{lemma65}
Let $g:\Lambda\backslash (\p\Lambda)\subset \bb T^d \to\bb R $ be a  function  which is continuous near $\p\Lambda$ with  an extension to $\Lambda$ which is also continuous near $\p\Lambda$. Then,
\begin{align}
& \int_{\p\Lambda}g(u^-)|\<\vec{\zeta}(u),e_j\>|\,dS(u)  \;=\; \lim_{N\to\infty}\frac{1}{N^{d-1}}\sum_{x\in\Gmenosj} g\big(\pfrac{x}{N}\big)  \quad \text{ and} \label{eq617}\\
&\int_{\p\Lambda}g(u^-)\<\vec{\zeta}(u),e_j\>\,dS(u)     \;=\;\lim_{N\to\infty}\frac{1}{N^{d-1}}\Bigg[\sum_{x\in\GmenosjR} g\big(\pfrac{x}{N}\big) - \sum_{x\in\GmenosjL} g\big(\pfrac{x}{N}\big)\Bigg] \,.\label{eq617b}
\end{align}
Analogously, if $g:\Lambda^\complement\subset \bb T^d \to\bb R $ 
is a  function  which is continuous near $\p\Lambda$ with  an extension to the closure of  $\Lambda^\complement$ which is also continuous near $\p\Lambda$, then
\begin{align}
& \int_{\p\Lambda}g(u^+)|\<\vec{\zeta}(u),e_j\>|\,dS(u)  \;=\; \lim_{N\to\infty}\frac{1}{N^{d-1}}\sum_{x\in\Gmaisj} g\big(\pfrac{x}{N}\big)  \quad \text{ and} \label{eq617AA}\\
&\int_{\p\Lambda}g(u^+)\<\vec{\zeta}(u),e_j\>\,dS(u)     \;=\;\lim_{N\to\infty}\frac{1}{N^{d-1}}\Bigg[\sum_{x\in\GmaisjR} g\big(\pfrac{x}{N}\big) - \sum_{x\in\GmaisjL} g\big(\pfrac{x}{N}\big)\Bigg] \,.\label{eq617bAA}
\end{align}
\end{lemma}
\begin{proof}
In view of the previous discussion, we claim that
\begin{equation}\label{eq123} 
\lim_{N\to\infty}\frac{1}{N^{d-1}}\sum_{x\in\Gmenosj}\frac{h\big(\pfrac{x}{N}\big)}{ \big|\big\<\vec{\zeta}\big(\uu(\pfrac{x}{N})\big),e_j\big\>\big|}\;=\; \int_{\p\Lambda}h(u^-)\,dS(u)\,.
\end{equation}
for any continuous function $h: \Lambda\to \bb R$ such that $h(u)=0$ on the set $\{u\in \p\Lambda: \<\vec{\zeta}(u),e_j\>=0\}$.
This is due to the fact that the sum in the left hand side of \eqref{eq617} is equal to a Riemann sum for the integral on the right hand side of \eqref{eqGdS} modulus a small error. To see this, it is enough to note that if $x\in\Gmenos$, then $x/N$  is at a distance less or equal than $1/N$ to $\p\Lambda$, and recall that $\Lambda$ is compact, thus any continuous function over $\Lambda$ is uniformly continuous.

 Consider now the function $h: \Lambda\to \bb R$ given by
\begin{equation*}
h(u)\;:=\; g(u) \, |\<\vec{\zeta}\big(\uu(u)\big),e_j\>|\,.
\end{equation*}
Since $\uu(u)=u$ for $u\in\p\Lambda$, we have that $h(u)=0$ on the set $\{u\in \p\Lambda: \<\vec{\zeta}(u),e_j\>=0\}$. 
 Then, considering this particular function $h$ in \eqref{eq123}  leads to \eqref{eq617}.
The limit \eqref{eq617b} can be derived from \eqref{eq617} 
noticing that,  for  $N$ sufficiently large,
\begin{itemize}
\item if $x\in  \GmenosjR$, then $ \<\vec{\zeta}\big(\uu(x/N)\big), e_j\>>0$ and 
\item if $x\in  \GmenosjL$, then $\<\vec{\zeta}\big(\uu(x/N)\big), e_j\> <0 $,
\end{itemize}   
see Figure~\ref{Fig5} for support. The proofs for \eqref{eq617AA} and \eqref{eq617bAA} are analogous.
\end{proof}

\begin{proof}[Proof of Proposition~\ref{prop62one}]
The fact that boundary integrals are not well-defined in the whole Skorohod space $\mc D([0,T ],\mc M)$ forbids us to directly apply Portmanteau's Theorem. To circumvent this  technical obstacle, fix $\varepsilon > 0$ which will be taken small later. Adding and subtracting the convolution of $\rho(t,u)$ with the approximation of identity $\iota_\eps$ defined in \eqref{approxidentity}, we bound the probability in \eqref{6.4} by the sum of 
\begin{equation} \label{6.5}
\begin{split}
&\bb Q^\beta_*\Big[\pi.:\sup_{0\leq t\leq T}\Big|\<\rho_t, H\>- \<\rho_0, H \>-\int_0^t\<\rho_s ,\Delta H\> \,ds \Big.\\
&-\int_0^t\int_{\p\Lambda}(\rho_s\ast\iota_\eps)(u^+)\sum_{j=1}^{d}\p_{u_j} H(u^+)\<\vec{\zeta}(u),e_j\>\,dS(u)ds\\
&+\int_0^t\int_{\p\Lambda} (\rho_s\ast\iota_\eps)(u^-)\sum_{j=1}^{d}\p_{u_j} H(u^-)\<\vec{\zeta}(u),e_j\>\,dS(u)ds\\
&+\int_0^t \int_{\p\Lambda}\alpha ((\rho_s\ast\iota_\eps)(u^-)-(\rho_s\ast\iota_\eps)(u^+))\\
&\hspace{2cm}\times (H(u^+)-H(u^-))\sum_{j=1}^d|\<\vec{\zeta}(u), e_j\>|\,dS(u)ds\Big|>\delta/2\Big]
\end{split}
\end{equation}
and 
\begin{equation}\label{second} 
\begin{split}
&\bb Q^\beta_*\Big[\pi.:\, \sup_{0\leq t\leq T}\Big|\int_0^t\int_{\p\Lambda}\Big((\rho_s\ast\iota_\eps)(u^+)-\rho_s(u^+)\Big)\sum_{j=1}^{d} H(u^+)\<\vec{\zeta}(u),e_j\>\,dS(u)ds\Big.\Big.\\
&-\int_0^t\int_{\p\Lambda}\Big((\rho_s\ast\iota_\eps)(u^-)-\rho_s(u^-)\Big)\sum_{j=1}^{d}\p_{u_j}H(u^-)\<\vec{\zeta}(u),e_j\>\,dS(u)ds\Big.\\
&-\int_0^t\int_{\p\Lambda}\alpha \Big((\rho_s\ast\iota_\eps)(u^-)-\rho_s(u^-)\Big)(H(u^+)-H(u^-))\sum_{j=1}^d|\<\vec{\zeta}(u), e_j\>|\,dS(u)ds\Big.\\
&+\!\int_0^t\!\!\int_{\p\Lambda}\!\!\!\alpha \Big((\rho_s\ast\iota_\eps)(u^+)-\rho_s(u^+)\Big)\\
&\hspace{2cm}\times(H(u^+)-H(u^-))\sum_{j=1}^d|\<\vec{\zeta}(u), e_j\>|\,dS(u)ds\Big|>\delta/2\Big].
\end{split}
\end{equation}
where $\iota_\eps$ and the convolution $\rho_s\ast\iota_\eps$ were defined in \eqref{convolution}. Adapting results of \cite[Chapter III]{Adams} to our context, the reader can check that functions in the Sobolev space $\sobH$ are continuous in $\bb T^d\backslash \p\Lambda$. Thus,  Lemma~\ref{Prop5.7} gives us that \eqref{second} vanishes as $\eps\to 0$.
It remains to deal with \eqref{6.5}.
By Portmanteau's Theorem, \eqref{6.5} is bounded from above by 
\begin{align*}
&\varlimsup_{N\rightarrow\infty}\bb Q^{\beta, N}_{\mu_N}\Big[\pi.:\sup_{0\leq t\leq T}\Big|\<\pi_t, H\>- \<\pi_0, H \>-\int_0^t\<\pi_s ,\Delta H\> \,ds\Big.\\
&-\int_0^t\int_{\p\Lambda}(\pi_s\ast\iota_\eps)(u^+)\sum_{j=1}^{d}\p_{u_j} H(u^+)\<\vec{\zeta}(u),e_j\>\,dS(u)ds\Big.\\
&+\int_0^t\int_{\p\Lambda}(\pi_s\ast\iota_\eps)(u^-)\sum_{j=1}^{d}\p_{u_j} H(u^-)\<\vec{\zeta}(u),e_j\>\,dS(u)ds\Big.\\
&+\!\!\int_0^t\!\!\int_{\p\Lambda}\hspace{-0.3cm}\alpha ((\pi_s\ast\iota_\eps)(u^-)\!-\!(\pi_s\ast\iota_\eps)(u^+))\\
&\hspace{2cm}\times(H(u^+)\!-\!H(u^-))\sum_{j=1}^d|\<\vec{\zeta}(u), e_j\>|\,dS(u)ds\Big|>\delta/2\Big]\!,
\end{align*}
since the supremum above is a continuous function in the Skorohod metric. 
Now, recalling that $\bb Q^{\beta, N}_{\mu_N}$ is the probability induced by $\bb P^{\beta}_{\mu_N}$ via the empirical measure, adding and subtracting $\<\pi_s^N, N^2 \bb L_N H\>$, adding and subtracting $\frac{1}{N^{d-1}}\sum_{x\in\Gamma_N}N\bb L_N H(\pfrac{x}{N})\eta_s^{\eps N}(x)$, applying \eqref{remark} and the Lemma~\ref{lemma65},  we can bound the previous expression by the sum of
\begin{equation}\label{eq614a}
\varlimsup_{N\rightarrow\infty}\bb P^{\beta}_{\mu_N}\Big[\sup_{0\leq t\leq T}\Big|\<\pi^N_t, H\> -\<\pi^N_0, H \> - \int_0^t \<\pi^N_s ,N^2\bb L_N H \> \,ds\Big| >\delta/8\Big]\,,
\end{equation}
\begin{equation}\label{eq614aa}
\varlimsup_{N\rightarrow\infty}\bb P^{\beta}_{\mu_N}\Big[\sup_{0\leq t\leq T}\Big|\int_0^t \sum_{x\notin\Gamma_N}\Big(N^2\bb L_N H\big(\pfrac{x}{N}\big)-\Delta H\big(\pfrac{x}{N}\big)\Big)\eta_s(x) \,ds\Big| >\delta/8\Big]\,,
\end{equation}
\begin{equation} \label{eqmaA}
\varlimsup_{N\rightarrow\infty}\bb P^{\beta}_{\mu_N}\Big[\sup_{0\leq t \leq T} \Big| \frac{1}{N^{d-1}}\int_0^t \sum_{x\in\Gamma_N}N\bb L_N H\big(\pfrac{x}{N}\big)(\eta_s(x)-\eta^{\varepsilon N}_s(x))\,ds\Big|>\delta/8\Big]
\end{equation} 
and
\begin{equation}\label{eq6.15a}
\begin{split}
&\varlimsup_{N\rightarrow\infty}\bb P^{\beta}_{\mu_N}\Big[\sup_{0\leq t\leq T}\Big|\int_0^t \sum_{x\in\Gamma_N}N\bb L_N H\Big(\frac{x}{N}\Big)\eta_s^{\eps N}(x) \,ds\\
&+\sum_{j=1}^{d}\int_0^t\frac{1}{N^{d-1}}\hspace{-0.2cm}\sum_{x\in \GmenosjR}\hspace{-0.2cm}\eta_s^{\varepsilon N}(x)\p_{u_j}H(\uu^-)\,ds\\
&-\sum_{j=1}^{d}\int_0^t\frac{1}{N^{d-1}}\hspace{-0.2cm}\sum_{x\in \GmenosjL}\hspace{-0.2cm}\eta_s^{\varepsilon N}(x)\p_{u_j}H(\uu^-)\,ds\\
&-\sum_{j=1}^{d}\int_0^t\vspace{-0.2cm}\frac{1}{N^{d-1}}\hspace{-0.2cm}\sum_{x\in \GmaisjR}\hspace{-0.2cm}\eta_s^{\varepsilon N}(x)\p_{u_j}H(\uu^+)\,ds\\
&+\sum_{j=1}^{d}\int_0^t\vspace{-0.2cm}\frac{1}{N^{d-1}}\hspace{-0.2cm}\sum_{x\in \GmaisjL}\hspace{-0.2cm}\eta_s^{\varepsilon N}(x)\p_{u_j}H(\uu^+)\,ds\\
& +\sum_{j=1}^d\int_0^t\!\!\frac{1}{N^{d-1}}\hspace{-0.2cm}\sum_{x\in\Gmenosj}\hspace{-0.2cm}\alpha\,\eta_s^{\varepsilon N}(x)(H(\uu^+)-H(\uu^-))\,ds \\
& -\sum_{j=1}^d\int_0^t\!\!\frac{1}{N^{d-1}}\hspace{-0.2cm}\sum_{x\in\Gmaisj}\hspace{-0.2cm}\alpha\,\eta_s^{\varepsilon N}(x)(H(\uu^+)-H(\uu^-))\,ds+ \text{err}(N) \Big|>\delta/8\Big]\,,
\end{split}
\end{equation}
where $\text{err}(N)$ is a error that goes in modulus  to zero as $N\to\infty$. Proposition~\ref{bm} tells us that  \eqref{eq614a} is null.
The approximation of the continuous Laplacian by the discrete Laplacian
assures that \eqref{eq614aa} is null. Since $N\bb L_N H$ is a sequence of uniformly bounded functions, Lemma \ref{replacement2} allows we conclude that  \eqref{eqmaA} vanishes as $\eps\searrow 0$. Finally, provided by formulas \eqref{sumnj2} and \eqref{sumnj3} and recalling the decomposition $\Gamma_N=\Gmais\cup \Gmenos$, we can see that, except for the error term, all terms inside the supremum in \eqref{eq6.15a} cancel. This concludes the proof.
\end{proof}

\subsection{Characterization of limit points for \texorpdfstring{$\beta\in(1,\infty]$}{beta>1}.} \label{6.3}
\begin{proposition}
Let $H\in C^2(\bb T^d\backslash \partial\Lambda)$. For all $\delta>0$,
\begin{equation} \label{6.11}
\begin{split}
\bb Q^\beta_*\Big[\pi.: \sup_{0\leq t\leq T}\Big|&\< \rho_t, H\> - \< \rho_0 , H\> - \int_0^t \< \rho_s , \Delta H \> \, ds\Big.\\
&-\int_0^t\int_{\p\Lambda}\rho_s(u^+)\sum_{j=1}^d\p_{u_j} H(u^+)\<\vec{\zeta},e_j\>\,dS(u)ds\\
&+\int_0^t\int_{\p\Lambda}\rho_s(u^-)\sum_{j=1}^d\p_{u_j} H(u^-)\<\vec{\zeta},e_j\>\,dS(u)ds\Big| >\delta\,\Big]\;=\;0.
\end{split}
\end{equation}
\end{proposition}
\begin{proof}
The proof of this proposition is similar, in fact, simpler than the one of Proposition~\ref{prop62one}. In this case,
\begin{equation}\label{sumnj2neumann}
\begin{split}
&\frac{1}{N^{d-1}}\sum_{x\in\Gmenos}N\bb L_N H\big(\pfrac{x}{N}\big)\eta^{\varepsilon N}_s(x)\\
 & \;=\;
\frac{1}{N^{d-1}}\sum_{j=1}^d \Bigg\{\sum_{x\in\GmenosjR}\Big[\alpha N^{1-\beta} \big(H(\uu^+)-H(\uu^-)\big)-\p_{u_j}H(\uu^-)\Big]\eta^{\varepsilon N}_s(x)\\
&\hspace{2cm}+\sum_{x\in\GmenosjL}\Big[\p_{u_j}H(\uu^-)+\alpha N^{1-\beta}\big(H(\uu^+)-H(\uu^-)\big)\Big]\eta^{\varepsilon N}_s(x)\Bigg\}\,.
\end{split}
\end{equation}
and 
\begin{equation}\label{sumnj3neumann}
\begin{split}
&\frac{1}{N^{d-1}}\sum_{x\in\Gmais}N\bb L_N H\big(\pfrac{x}{N}\big)\eta^{\varepsilon N}_s(x)\\
 & \;=\;
\frac{1}{N^{d-1}}\sum_{j=1}^d \Bigg\{\sum_{x\in\GmaisjR}\Big[\p_{u_j}H(\uu^+)+\alpha N^{1-\beta} \big(H(\uu^-)-H(\uu^+)\big)\Big]\eta^{\varepsilon N}_s(x)\\
&\hspace{2cm}+\sum_{x\in\GmaisjL}\Big[\alpha N^{1-\beta}\big(H(\uu^-)-H(\uu^+)\big)-\p_{u_j}H(\uu^+)\Big]\eta^{\varepsilon N}_s(x)\Bigg\}\,.
\end{split}
\end{equation}
 Since $\beta\in (1,\infty]$, we conclude that all terms above  involving $\alpha$ disappear in the limit as $N\to \infty$. Noting that there are no surface integrals in \eqref{6.11} involving $\alpha$, it is a simple game to repeat the steps in the proof of Proposition~\ref{prop62one} to finally conclude \eqref{6.11}.
\end{proof}

\section{Uniqueness of weak solutions} \label{s7}
The hydrodynamic equation \eqref{edpheat} is the classical heat equation, which does not need any consideration about uniqueness of weak solutions. Thus, we only need  to guarantee that weak solutions of \eqref{edp12} and \eqref{edpbc} are unique. 

Let us trace the strategy for the proof of uniqueness, which works for both \eqref{edp12} and \eqref{edpbc}. 
Considering in each case $\beta=1$ or $\beta\in(1,\infty]$  a suitable set of test functions, we can annul all surface integrals. Being more precise, consider the following definitions:
\begin{definition}\label{operator1}  Let $\Dr\subset L^2(\bb T^d)$ be the set of functions  $H:\bb T^d\to \bb R$ such that $H(u)=h_1(u){\bf 1}_\Lambda(u)+h_2(u){\bf 1}_\Lambda^\complement(u)$, where
\begin{enumerate}[(i)]
\item $h_i\in C^2(\bb T^d)$ for $i=\{1,2\}$.
\item $\<\nabla h_1(u),\vec{\zeta}(u)\>=\<\nabla h_2(u),\vec{\zeta}(u)\>=\big(h_2(u)-h_1(u)\big)\dd\sum_{j=1}^d|\<\vec{\zeta}(u), e_j\>|$ ,  $\forall u\in\p\Lambda$.

\end{enumerate}
Define  the operator  $\Lr: \Dr\rightarrow L^{2}(\bb T^d)$ by
\begin{equation*}
\Lr H(u)\,=\;\left\{\begin{array}{cl}
\Delta h_1(u), \,\, &  \mbox{if}\,\,\,\,u\in\Lambda\,,\smallskip\\
\Delta h_2(u), \,\, &  \mbox{if}\,\,\,\,u\in\Lambda^\complement \,.
\end{array}
\right.
\end{equation*}
\end{definition}

\begin{definition}\label{operator}  Let $\Dn\subset L ^2(\bb T^d)$ be the set of functions  $H:\bb T^d\to \bb R$ such that $H(u)=h_1(u){\bf 1}_\Lambda(u)+h_2(u){\bf 1}_\Lambda^\complement(u)$, 
where:
\begin{enumerate}[(i)]
\item $h_i\in C^2(\bb T^d)$ for $i=\{1,2\}$.\smallskip
\item $\<\nabla h_1(u),\vec{\zeta}(u)\>=\<\nabla h_2(u),\vec{\zeta}(u)\>=0$ , $\forall u\in\p\Lambda$.\smallskip
\end{enumerate}
Define  the operator   $\Ln: \Dn\rightarrow L^{2}(\bb T^d)$ by
\begin{equation*}
\Ln H(u)\;=\;\left\{\begin{array}{cl}
\Delta h_1(u) \,\, &  \mbox{if}\,\,\,\,u\in\Lambda\,,\smallskip\\
\Delta h_2(u) \,\, &  \mbox{if}\,\,\,\,u\in\Lambda^\complement \,.
\end{array}
\right.
\end{equation*}
\end{definition}
It is straightforward to check that, if $\rho$ is a weak solution of \eqref{edp12}, then
\begin{equation}\label{71}
\begin{split}
&\< \rho_t, H\> - \< \rho_0 , H\> - \int_0^t\! \< \rho_s , \Lr H \> \, ds\;=\; 0\,,\quad \forall H\in \Dr\,,\,\forall t\in[0,T]\,,
\end{split}
\end{equation}
while, if $\rho$ is a weak solution of \eqref{edpbc}, then
\begin{equation}\label{72}
\begin{split}
&\< \rho_t, H\> - \< \rho_0 , H\> - \int_0^t\! \< \rho_s , \Ln H \> \, ds\;=\; 0\,,\quad \forall H\in \Dn\,, \,\forall t\in[0,T]\,.
\end{split}
\end{equation}
In both cases, if an orthonormal basis of $L^2(\bb T^d)$ composed of eigenfunctions for the corresponding operator (associated to nonpositive eigenvalues) is available, this would easily lead to the proof of uniqueness, as we shall see later.   However, this is not the case. So, to overcome this situation  we extend the corresponding operator via a \textit{Friedrichs extension} (see \cite{z} on the subject) to achieve the desired orthonormal basis.

Let us briefly explain the notion of Friedrichs extension. Let $X$ be a Hilbert space and denote by $\<\cdot,\cdot\>$ and $\Vert\cdot\Vert$ its inner product and norm, respectively. Consider a linear, strongly monotone and symmetric operator $\mathcal{A}:\mf D\subset X\rightarrow X$, where by \textit{strongly monotone} we mean that there exists  $c>0$ such that 
\begin{equation*}
\< \mathcal{A} H, H\>\;\geq\; c \Vert H\Vert^2\,,\quad \forall \,H\in\mf D\,.
\end{equation*}
Denote by $\< \cdot, \cdot\>_{\mc E(\mathcal{A})}$ the so-called \textit{energetic} inner product on $\mf D$ associated to $\mc A$, which is defined by
\begin{equation*}
\< F,G\>_{\mc E(\mathcal{A})} \; :=\;  \< F,\,\mathcal{A}  G\>\,.
\end{equation*}
Let $\HF$ be the set of all functions $F$ in $X$ for which there exists a sequence $\{F_n : n\ge 1\}$ in $\mf D$ such that $F_n$ converges to $F$ in $X$ and $F_n$ is Cauchy for the inner product $\< \cdot, \cdot \>_{\mc E(\mc A)}$. A sequence $\{F_n : n\ge 1\}$ with these properties will be called an \textit{admissible sequence} for $F$. For $F$, $G$ in $\HF $, let
\begin{equation}\label{innerproduct}
\< F,G\>_{\text{\rm Fried}}\; :=\; \lim_{n\to\infty} \< F_n,G_n\>_{\mc E(\mathcal{A})}\,,
\end{equation}
where $\{F_n : n\ge 1\}$, $\{G_n : n\ge 1\}$ are admissible sequences for $F$ and $G$, respectively. 
 By \cite[Proposition 5.3.3]{z}, the limit exists and does not depend on the admissible sequence chosen and, moreover, the space $\HF$ endowed with the scalar product $\< \cdot, \cdot \>_{\text{\rm Fried}}$  is a real Hilbert space, usually called the \textit{energetic space} associated to $\mc A$. 
 
 The Friedrichs extension $\mathcal{A}_{\text{\rm Fried}}:\DF\to X$ of the operator $\mathcal{A}$ is then defined as follows.  Let $\DF$ be the set of vectors in $F\in\HF$ for which there exists a vector $f\in X$ such that
 \begin{equation*}
\<  F,G\>_{\text{\rm Fried}} \;=\; \<  f,G\>\,, \quad \forall G\in \HF\,. 
\end{equation*}
and let $\AF F=f$. See the excellent book \cite{z} for why this operator  $\mathcal{A}_{\text{\rm Fried}}:\DF\to X$ is indeed an extension of $\mc A:\mf D\to X$ and  more details on the construction.  The main result about Friedrichs extensions and eigenfunctions we cite here is the next one. 
\begin{theorem}[\cite{z}, Theorem 5.5C]\label{proper} Let $\mathcal{A}: \mf D\subseteq X\rightarrow X$ be a linear, symmetric and strongly monotone operator and let $\mathcal{A}_{\text{\rm Fried}}: \mf D_{\text{\rm Fried}}\subseteq X\rightarrow X$ be its Friedrichs extension.  Assume additionally that the embedding $\HF \hookrightarrow X$ is compact. Then,
\begin{itemize}
\item[(a)] The eigenvalues of $-\mathcal{A}_{\text{\rm Fried}}$ form a countable set $0<c\leq  \mu_1\le\mu_2\le \cdots$ with $\lim_{n\to\infty} \mu_n=\infty$, and all these eigenvalues have finite multiplicity.
\item[(b)] There exists a complete orthonormal basis of $X$ composed of eigenvectors of $\mc A_{\text{\rm Fried}}$.
\end{itemize}
\end{theorem}
Denote by $\bb I$ the identity operator. If $\mf L: \mf D\subseteq X\rightarrow X$ is a symmetric nonpositive  operator, then  
$\bb I -\mf L:\mf D\rightarrow X$ is  symmetric and strongly monotone with $c=1$. In fact,
\begin{equation*}
\<(\bb I-\mf L) H, H\>\;=\; \Vert H\Vert^2 + \<-\mf L H, H\>\;\geq\; \Vert H\Vert^2\,,\quad \forall\,H\in \mf D\,.
\end{equation*}
Therefore, under the hypothesis that $\mf L: \mf D\subseteq X\rightarrow X$ is a symmetric and nonpositive linear operator, we may consider the  Friedrichs extension of ${(\bb I -\mf L)}$.   
\begin{proposition}\label{prop72}
Let $\mf L: \mf D\subseteq X\rightarrow X$ be a symmetric nonpositive  operator.  Denote by $(\bb I -\mf L)_{\text{\rm Fried}}: \mf D_{\text{\rm Fried}}\rightarrow X$  the Friedrichs extension of $(\bb I -\mf L):\mf D\to X$ and by $\HF$ the corresponding energetic space. Assume that the embedding $\HF \hookrightarrow X$ is compact. 
Then, there exists at most one measurable  function $\rho:[0,T] \to X$ such that
\begin{equation}\label{sup}
\sup_{t\in[0,T]}\Vert \rho_t\Vert\;<\;\infty\
\end{equation}
and
\begin{equation*}
\begin{split}
\< \rho_t, H\> - \< \rho_0 , H\> - \int_0^t \< \rho_s , \mf L H \> \, ds\;=\; 0\,,\quad \forall H\in \mf D\,, \,\forall t\in[0,T]\,.
\end{split}
\end{equation*}
where  $\rho_0$ is a fixed element of $X$.
\end{proposition}
\begin{proof}
Consider $\rho^1, \rho^2$ two solutions of above and write  $\rho=\rho^1-\rho^2$. By linearity, 
\begin{equation*}
\< \rho_t, H\>  - \int_0^t \< \rho_s , \mf L H \> \, ds\;=\; 0\,,\quad \forall H\in \mf D\,, \,\forall t\in[0,T]\,.
\end{equation*}
which is the same as
\begin{equation*}
\< \rho_t, H\>  + \int_0^t \< \rho_s , (\bb I-\mf L) H \> \, ds- \int_0^t \< \rho_s ,  H \> \, ds\;=\; 0\,,\quad \forall H\in \mf D\,, \,\forall t\in[0,T]\,.
\end{equation*}
Since $\DF\subseteq \HF$, the last equation  can be extended to
\begin{equation}\label{pth}
\< \rho_t, H\>  + \int_0^t \< \rho_s , (\bb I-\mf L)_{\text{\rm Fried}} H \> \, ds- \int_0^t \< \rho_s ,  H \> \, ds\;=\; 0\,,\quad \forall H\in \DF\,, \,\forall t\in[0,T]\,.
\end{equation}
By Theorem \ref{proper}, the Friedrichs extension $(\bb I-\mf L)_{\text{\rm Fried}} : \mf D_{\text{\rm Fried}}  \to X$ has eigenvalues $1\le \lambda_1 \le \lambda_2 \le \cdots$, all of them having finite multiplicity with $\lim_{n\to\infty}\lambda_n=\infty$,  and there exists a complete orthonormal basis $\{\Psi_j\}_{i\in \bb N}$ of $L^2(\bb T^d)$ composed of eigenfunctions. 
Denote $$\mf L_{\text{\rm Fried}}\;:=\;\bb I-(\bb I-\mf L)_{\text{\rm Fried}}\,.$$
Thus, $\{\Psi_j\}_{j\in \bb N}$ is also a set of eigenfunctions for the operator $\mf L_{\text{\rm Fried}}$  whose eigenvalues are given by  $\mu_j=1-\lambda_j\leq 0$. 
 Define 
\[R(t)\;=\;\sum_{j=1}^\infty\frac{1}{j^{2}(1-\mu_j)}\<\rho_t,\Psi_j\>^{2}\quad \text{ for } t\in [0,T]\,.\]   Since $\rho$ satisfy \eqref{pth}, we have that 
\begin{equation}\label{derivada}
\frac{d}{dt}\< \rho_t,\Psi_j\>^{2}\;=\;2\<\rho_t,\Psi_j\>\<\rho_t,\mf L_{\text{\rm Fried}}\Psi_j\>\;=\;2\mu_j\<\rho_t,\Psi_j\>^2\,.
\end{equation}
By \eqref{sup} and the Cauchy-Schwarz inequality, we have that 
\begin{equation*}
\sum_{j=1}^\infty\frac{2|\mu_j|}{j^{2}(1-\mu_j)}\<\rho_t,\Psi_j\>^{2}\;\leq\;\sum_{j=1}^\infty\frac{2|\mu_j|}{j^{2}(1-\mu_j)}\Big(\sup_{t\in [0,T]}\Vert \rho_t\Vert^{2}\Big)\;<\;\infty\,,
\end{equation*}
which together with \eqref{derivada} implies  that
\begin{equation*}
\frac{d}{dt}R(t)\;=\;\sum_{j=1}^\infty\frac{2\mu_j}{j^{2}(1-\mu_j)}\< \rho_t,\Psi_j\> ^{2}\;\leq\;0\,.
\end{equation*}
 Since $R(t)\geq0$,  $R(0)=0$, and $dR/dt\leq 0$, we conclude that $R(t)=0$ for all $t\in[0,T]$ and hence $\<\rho_t,\Psi_j\>^2=0$ for any $t\in[0,T]$. 
Due to $\{\Psi_j\}_{j\in \bb N}$ be a complete orthonormal basis of $X$, we deduce that $\rho\equiv0$, finishing the proof.
\end{proof}

In view of \eqref{71} and \eqref{72}, considering $X$ as the Hilbert space $L^2(\bb T^d)$ and applying the last proposition, to achieve the uniqueness of weak solutions of \eqref{edp12} and \eqref{edpbc} it is enough to assure that 
\begin{enumerate}
\item The operators $\bb I-\Lr:\Dr\subseteq L^2(\bb T^d)\to L^2(\bb T^d)$ and $\bb I-\Ln:\Dn\subseteq L^2(\bb T^d)\to L^2(\bb T^d)$ 
are symmetric nonpositive linear operators. \smallskip
\item Denoting by $\HrF$ and $\HnF$ their respective energetic spaces, the embeddings $\HrF \hookrightarrow L^2(\bb T^d)$ and $\HnF \hookrightarrow L^2(\bb T^d)$ are compact.
\end{enumerate}\medskip 

This is precisely what we are going to do in the next four propositions. Denote by $\cev{\zeta}(u)=-\vec{\zeta}(u)$ the normal exterior vector to the region $\Lambda^\complement$ at  $u\in \p\Lambda$. Recall that $\<\cdot,\cdot\>$ is used for both the inner products in $L^2(\bb T^d)$ and in $\bb R^d$.
\begin{proposition} \label{simetrico1}
The operator $-\Lr:\Dr\rightarrow L^2(\bb T^d)$ is symmetric and nonnegative.
\end{proposition}
\begin{proof}
Let $H, G\in \Dr$. We can write
$H=h_1{\bf 1}_\Lambda+h_2{\bf 1}_\Lambda^\complement$ and $ G=g_1{\bf 1}_\Lambda+g_2{\bf 1}_\Lambda^\complement$, where $h_1, h_2, g_1, g_2 \in C^2(\bb T^d)$. By the third Green identity (see Appendix~\ref{appendix}, Theorem \ref{evans}), \begin{equation*}
\int_{\bb T^d} \big(h\Delta g - g \Delta h\big) \,du\;=\;\int_{\p\Lambda}\Big(h\<\nabla g,\vec{\zeta}\,\>-g\<\nabla h,\vec{\zeta}\,\>\Big) \,dS\,,
\end{equation*} 
where $dS$ is an infinitesimal volume element of $\p\Lambda$. Thus, 
\begin{align*}
\<H,-\Lr G\>=& \<h_1{\bf 1}_\Lambda+h_2{\bf 1}_{\Lambda^\complement}, -\Delta g_1{\bf 1}_\Lambda-\Delta g_2{\bf 1}_{\Lambda^\complement} \>\\
=& -\int_{\Lambda}h_1\Delta g_1\, du - \int_{\Lambda^\complement} h_2\Delta g_2\, du\\
=& -\int_{\Lambda}g_1\Delta h_1\, du -\int_{\p\Lambda}\Big(h_1\<\nabla g_1,\vec{\zeta}\,\>-g_1\<\nabla h_1,\vec{\zeta}\,\>\Big)\, dS\\
&-\int_{\Lambda^\complement} g_2\Delta h_2\, du-\int_{\p\Lambda^\complement}\Big(h_2\<\nabla g_2,\cev{\zeta}\,\>-g_2\<\nabla h_2,\cev{\zeta}\,\>\Big)\, dS\\
=& -\int_{\Lambda}g_1\Delta h_1\, du -\int_{\p\Lambda}\Big(h_1\<\nabla g_1,\vec{\zeta}\,\>-g_1\<\nabla h_1,\vec{\zeta}\,\>\Big)\, dS\\
&-\int_{\Lambda^\complement} g_2\Delta h_2\, du-\int_{\p\Lambda^\complement}\Big(g_2\<\nabla h_2,\vec{\zeta}\,\>-h_2\<\nabla g_2,\vec{\zeta}\,\>\Big)\, dS\,.
\end{align*} 
Using the boundary condition  in the item (ii) of Definition~\ref{operator1} and $\p\Lambda^\complement=\p\Lambda$, we conclude that the last expression above is equal to
\begin{align*}
&-\int_{\Lambda}g_1\Delta h_1\, du- \int_{\Lambda^\complement} g_2\Delta h_2\, du\\
&-\int_{\p\Lambda}\!\!\Big((h_1\!-\!h_2)\sum_{j=1}^d|\<\vec{\zeta}, e_j\>|(g_2\!-\!g_1)\!-\!(g_1-g_2)\sum_{j=1}^d|\<\vec{\zeta}, e_j\>|(h_2-h_1)\Big)dS\\
&=-\int_{\Lambda}g_1\Delta h_1\, du- \int_{\Lambda^\complement} g_2\Delta h_2\, du\,.
\end{align*}
Then, $
\<H,-\Lr G\>=-\int_{\Lambda}g_1\Delta h_1\, du - \int_{\Lambda^\complement} g_2\Delta g_2\, du=\, \<-\Lr H, G\>$. 
For the nonnegativeness, note that
\begin{align*}
&\<H, -\Lr H\>\;=\; - \int_{\Lambda}h_1 \Delta h_1\,du-\int_{\Lambda^\complement}h_2 \Delta h_2\,du\\
&=\int_{\Lambda}|\nabla h_1|^2\,du+\int_{\Lambda^\complement}|\nabla h_2|^2\, du-\int_{\p\Lambda}\Big(\<\nabla h_1,\vec{\zeta}\,\>h_1+\<\nabla h_2, \vec{\zeta}\,\>h_2\Big) \,dS
\end{align*}
where the second equality above holds by the second Green identity, see Appendix, Theorem \ref{evans}, and $\p(\Lambda^{\complement})=\p\Lambda$.
Since $\int_{\Lambda}|\nabla h_i|^2\,du\geq0$, for $i=1,2$, it is enough to check that $-\int_{\p\Lambda}\Big(\<\nabla h_1, \vec{\zeta}\,\>h_1+\<\nabla h_2, \vec{\zeta}\,\>h_2 \Big)\,dS\;\geq\;0$. In fact, 
\begin{align*}
&-\int_{\p\Lambda}\Big(\<\nabla h_1, \vec{\zeta}\,\>h_1+\<\nabla h_2, \cev{\zeta}\,\>h_2\Big) \,dS\; =\;-\int_{\p\Lambda}\Big(\<\nabla h_1, \vec{\zeta}\,\>h_1-\<\nabla h_2, \vec{\zeta}\,\>h_2 \Big)\,dS\\
&=\!\int_{\p\Lambda}\sum_{j=1}^d|\<\vec{\zeta},e_j\>|\Big((h_2-h_1)h_2-(h_2-h_1)h_1\Big) dS\\
&=2\int_{\p\Lambda}\sum_{j=1}^d|\<\vec{\zeta},e_j\>|(h_2-h_1)^2\, dS\;\geq\; 0\,,
\end{align*}
where  the second equality holds by item (ii) of Definition \ref{operator1}.
\end{proof}

\begin{proposition}\label{embedding1}
The embedding $\HrF \hookrightarrow L^2(\bb T^d)$ is compact.
\end{proposition}
\begin{proof}
Let $\{ H_n\}$ be a bounded sequence in $\HrF $. Fix $\{ F_n\}$ a sequence in $\Dr$ such that $\Vert F_n-H_n\Vert \to 0$ when $n\rightarrow\infty$ and $\{ F_n\}$ is also bounded in $\HrF$. Thus, to show the compact embedding we need prove that $\{H_n\}$ have a convergent subsequence in $L^2(\bb T^d)$.  To get a convergent subsequence of $\{ H_n\}$, it is sufficient to find a convergent subsequence of $\{F_n\}$ in  $L^2(\bb T^d)$. Write $F_n=f_n\mathbf{1}_\Lambda+\tilde{f_n}\mathbf{1}_{\Lambda^\complement}$, with $f_n, \tilde{f_n} \in C^2(\bb T^d)$. 
Then,
\begin{align*}
&\< F_n,F_n\>_{\mc E(\bb I-\Lr)}=\<F_n,F_n\>+\<F_n, -\Lr F_n\>\\
&=\<f_n\mathbf{1}_\Lambda+\tilde{f_n}\mathbf{1}_{\Lambda^\complement},f_n\mathbf{1}_\Lambda+\tilde{f_n}\mathbf{1}_{\Lambda^\complement}\>+ \<f_n\mathbf{1}_\Lambda+\tilde{f_n}\mathbf{1}_{\Lambda^\complement},-\Delta f_n\mathbf{1}_{\Lambda}-\Delta \tilde{f_n}\mathbf{1}_{\Lambda^\complement} \>\,.
\end{align*}
Expanding the right hand side of above and using Green identity (see  Appendix~\ref{appendix}, Theorem \ref{evans}), we get that
\begin{align*}
 &\int_\Lambda f_n^2\,du + \int_{\Lambda^\complement}\tilde{f_n}^2\,du - \int_\Lambda f_n\Delta f_n\,du -\int_{\Lambda^\complement}\tilde{f_n}\Delta \tilde{f_n}\,du\\
&=\Vert f_n \mathbf{1}_{\Lambda}\Vert^2+\Vert \tilde{f_n} \mathbf{1}_{\Lambda^\complement}\Vert^2+\Vert \nabla f_n \mathbf{1}_{\Lambda}\Vert^2+\Vert \nabla \tilde{f_n} \mathbf{1}_{\Lambda^\complement}\Vert^2\\
&+2\int_{\p\Lambda}\sum_{j=1}^d|\<\vec{\zeta}, e_j\>|(f_n-\tilde{f_n})^2 dS\,.
\end{align*}
Under the hypotheses of boundedness of the sequence $\{F_n\}$  in the norm induced by $\<\cdot,\cdot\>_{\mc E(\bb I-\Lr)}$, the sequences $\{\Vert f_n \mathbf{1}_{\Lambda}\Vert^2\}$, $\{\Vert \tilde{f_n} \mathbf{1}_{\Lambda^\complement}\Vert^2\}$, $\{\Vert \nabla f_n \mathbf{1}_{\Lambda}\Vert^2\}$ and $\{\Vert \nabla \tilde{f_n} \mathbf{1}_{\Lambda^\complement}\Vert^2\}$ are bounded. By the Rellich-Kondrachov compactness theorem (see \cite[Theorem 5.7.1]{e}), $\{f_n \mathbf{1}_{\Lambda}\}$, $\{\tilde{f}_n\mathbf{1}_{\Lambda^\complement}\}$ have a common convergent subsequence in $L^2(\bb T^d)$. This implies that $\{F_n\}$ has a convergent subsequence. 
\end{proof}
\begin{proposition} \label{simetrico}
The operator $-\Ln:\Dn\rightarrow L^2(\bb T^d)$ is symmetric and nonnegative.
\end{proposition}
\begin{proof}
Let $H, G\in \Dn$. We can write 
$H=h_1{\bf 1}_\Lambda+h_2{\bf 1}_\Lambda^\complement$ and $ G=g_1{\bf 1}_\Lambda+g_2{\bf 1}_\Lambda^\complement$, where $h_1, h_2, g_1, g_2 \in C^2(\bb T^d)$. By the third Green identity, see Appendix~\ref{appendix}, Theorem \ref{evans}, we have that
\begin{equation*}
\int_{\bb T^d} h\Delta g \,du- g \Delta h \,du\;=\;\int_{\p\Lambda}h\<\nabla g,\vec{\zeta}\,\>-g\<\nabla h,\vec{\zeta}\,\> \,dS\;=\;0\,,
\end{equation*} where $dS$ is the infinitesimal volume element of $\p\Lambda$. Thus, 
\begin{align*}
& \<H,-\Ln G\>=\<h_1{\bf 1}_\Lambda+h_2{\bf 1}_{\Lambda^\complement}, -\Delta g_1{\bf 1}_\Lambda-\Delta g_2{\bf 1}_{\Lambda^\complement} \>\\
&= -\!\int_{\Lambda}h_1\Delta g_1 du - \!\int_{\Lambda^\complement} h_2\Delta g_2 du= -\!\int_{\Lambda}g_1\Delta h_1 du -\! \int_{\Lambda^\complement} g_2\Delta g_2 du= \<-\Ln H, G\>.
\end{align*}
For  nonnegativeness, 
\begin{align*}
\<H, -\mc L_{\Lambda}H\>&= - \int_{\Lambda}h_1 \Delta h_1\,du- \int_{\Lambda^\complement}h_2 \Delta h_2\,du=\int_{\Lambda}|\nabla h_1|^2\,du+\int_{\Lambda^\complement}|\nabla h_2|^2\, du\geq0,
\end{align*}
where the second equality above holds due to the second Green identity, see Appendix~\ref{appendix}, Theorem \ref{evans}.
\end{proof}

\begin{lemma}\label{embedding}
The embedding $\HnF \hookrightarrow L^2(\bb T^d)$ is compact.
\end{lemma}
\begin{proof}
Let $\{ H_n\}$ be a bounded sequence in $\Hn$. Fix a sequence $\{ F_n\}$  of functions in $\Dn$ 
such that $\Vert F_n-H_n\Vert \to 0$ when $n\rightarrow\infty$ and $\{ F_n\}$ is also bounded in $\HnF$. Thus, to show the compact embedding we need to prove that $\{H_n\}$ has a convergent subsequence in $L^2(\bb T^d)$.  To get a convergent subsequence of $\{ H_n\}$,
 it is sufficient to find a convergent subsequence of $\{F_n\}$ in  $L^2(\bb T^d)$. Write $F_n=f_n\mathbf{1}_\Lambda+\tilde{f_n}\mathbf{1}_{\Lambda^\complement}$, with $f_n\in C^2(\bb T^d)$. 
Then,
\begin{align*}
&\< F_n,F_n\>_{\mc E(\bb I-\Ln)}=\<F_n,F_n\>+\<F_n, -\Ln F_n\>\\
&=\<f_n\mathbf{1}_\Lambda+\tilde{f_n}\mathbf{1}_{\Lambda^\complement},f_n\mathbf{1}_\Lambda+\tilde{f_n}\mathbf{1}_{\Lambda^\complement}\>+ \<f_n\mathbf{1}_\Lambda+\tilde{f_n}\mathbf{1}_{\Lambda^\complement},-\Delta f_n\mathbf{1}_{\Lambda}-\Delta \tilde{f_n}\mathbf{1}_{\Lambda^\complement} \>.
\end{align*}
Expanding the right hand side and using Green identity, see Appendix~\ref{appendix}, Theorem \ref{evans}, we get that
\begin{align*}
 &\int_\Lambda f_n^2\,du + \int_{\Lambda^\complement}\tilde{f_n}^2\,du - \int_\Lambda f_n\Delta f_n\,du -\int_{\Lambda^\complement}\tilde{f_n}^2\Delta \tilde{f_n}\,du\\
&=\Vert f_n \mathbf{1}_{\Lambda}\Vert^2+\Vert \tilde{f_n} \mathbf{1}_{\Lambda^\complement}\Vert^2+\Vert \nabla f_n \mathbf{1}_{\Lambda}\Vert^2+\Vert \nabla \tilde{f_n} \mathbf{1}_{\Lambda^\complement}\Vert^2\,.
\end{align*}
Under the hypotheses of boundedness of the sequence $\{F_n\}$ in the norm induced by $\<\cdot,\cdot\>_{\mc E(\bb I-\Ln)}$, the sequences $\{\Vert f_n \mathbf{1}_{\Lambda}\Vert^2\}$, $\{\Vert \tilde{f_n} \mathbf{1}_{\Lambda^\complement}\Vert^2\}$, $\{\Vert \nabla f_n \mathbf{1}_{\Lambda}\Vert^2\}$ and $\{\Vert \nabla \tilde{f_n} \mathbf{1}_{\Lambda^\complement}\Vert^2\}$ are bounded. By the Rellich-Kondrachov Compactness Theorem,  $\{f_n \mathbf{1}_{\Lambda}\}$, $\{\tilde{f}_n\mathbf{1}_{\Lambda^\complement}\}$ have a common convergent subsequence in $L^2(\bb T^d)$. This implies that $\{F_n\}$ has a convergent subsequence. 
\end{proof}

\appendix
\section{Auxiliary results}\label{appendix}
\begin{proposition}[\cite{fgn1}] \label{A.3}
Let $G_1, \, G_2, \, G_3$ are continuous functions defined on the torus d-dimensional $\bb T^d$.  Then, the application from $D([0,T],\mc M)$ to $\bb R$ that associates to a trajectory $\{\pi_t:0\leq t\leq T\}$ the number
\begin{equation*}
\sup_{0\leq t \leq T}\Big|\<\pi_t, G_1\>-\<\pi_0, G_2\>-\int_0^t\<\pi_s, G_3\> \, ds\Big|
\end{equation*}
is continuous in the Skorohod metric of $D([0,T],\mc M)$.
\end{proposition}

\begin{theorem}[Green's formulas, see for instance Appendix C of \cite{e}] \label{evans}
Let $u,v\in C^2(\bar{U})$, where   $U$ is a  bounded open subset of $\bb R^n$, and $\p U$ is $C^1$. Denote by $\cdot $ the inner product in $\bb R^n$, and by $\nu$ the normal exterior unitary  vector to $U$ at $\p U$. Then,
\begin{enumerate}
	\item[(i)] $\int_U \Delta u dx \;= \;\int _{\p U}\frac{\p u}{\p \nu}\, dS$,
	\item[(ii)] $\int_U \nabla v\cdot \nabla u \, dx\;=\;-\int_U u\Delta v\,dx+ \int_{\p U}\frac{\p u}{\p \nu}u\, dS$,
	\item[(iii)] $\int_U u\Delta v-v\Delta u\, dx\;=\; \int _{\p U}u\frac{\p u}{\p \nu}-v\frac{\p u}{\p \nu}\, dS$. 
\end{enumerate}

\end{theorem}

\section*{Acknowledgements}
T.F. was supported through a project Jovem Cientista-9922/2015, FAPESB-Brazil. M.T. would like to thank CAPES for a PhD scholarship, which supported her research.

\bibliographystyle{plain}
\bibliography{bibliografia}

\end{document}